\documentclass[a4paper,11pt]{article}
\usepackage{amsmath,amsfonts,amssymb,amsthm,mathtools}
\usepackage[english]{babel}
\usepackage{hyperref}
\usepackage{inputenc}
\usepackage{csquotes}
\usepackage{fontenc}
\newcommand{\beq}{\begin{equation}}
\newcommand{\bey}{\begin{eqnarray}}
\newcommand{\beyy}{\begin{eqnarray*}}
\newcommand{\eeq}{\end{equation}}
\newcommand{\eey}{\end{eqnarray}}
\newcommand{\eeyy}{\end{eqnarray*}}
\newcommand{\N}{{\mathbb N}}

\newcommand{\C}{{\mathbb C}}

\newcommand{\eps}{\varepsilon}

\newtheorem{theorem}{Theorem}[section]
\newtheorem{lemma}[theorem]{Lemma}
\newtheorem{proposition}[theorem]{Proposition}
\newtheorem{remark}[theorem]{Remark}

\newtheorem{hypothesis}[theorem]{Hypothesis}

\newtheorem{corollary}[theorem]{Corollary}

\numberwithin{equation}{section}

\begin{document}

\title{\bf \Large An optimal uniqueness result for Riccati equations arising in abstract parabolic control problems}

\author{Paolo Acquistapace \footnote{address: Paolo Acquistapace, Universit\`a di Pisa ({\em retired}), Dipartimento di Matematica, Largo Bruno Pontecorvo 5, 56127 Pisa, ITALY; mail: paolo.acquistapace(at)unipi.it} and Francesco Bartaloni \footnote{address: Francesco Bartaloni, Universit\`a di Parma, Dipartimento di Matematica, Parco Area delle Scienze 53/A, 43124 Parma, ITALY; mail: francesco.bartaloni(at)unipr.it}}

\date{\empty}

\maketitle
\begin{abstract}
An abstract nonautonomous parabolic linear-quadratic regulator problem with very general final cost operator $P_{T}$ is considered, subject to the same assumptions under which a classical solution of the associated differential Riccati equation was shown to exist, in two papers appeared in 1999 and 2000, by Terreni and the first named author. We prove an optimal uniqueness result for the integral Riccati equation in a wide and natural class, filling a gap existing in the autonomous case, too. In addition, we give a regularity result for the optimal state. 
\end{abstract}
\section*{0\quad Introduction}
Let $H$, $U$ be Hilbert spaces. For fixed $T>0$ we consider, for every $s\in [0,T[\,$, the following regulator problem: minimize the functional
\beq \label{Js}
J_s(u) = \int_s^T \|M(r)^{1 / 2}y(r)\|_H^2\,dr + \int_s^T \|N(r)^{1 / 2}u(r)\|_U^2\,dr + \|P_T^{1/2}y(T)\|_H^2
\eeq
among all controls $u\in L^2(s,T;U)$, constrained by the following state equation, whose strong form is
\beq \label{state_eq_strong} 
\left\{ \begin{array}{l} y'(t)=A(t)y(t)-A(t)G(t)u(t), \quad t\in \,]s,T],\\[2mm]
y(s)=x, \end{array}\right.
\eeq
and whose mild form is
\beq \label{state_eq_mild} 
y(t) = U(t,s)x - \int_s^t U(t,r)A(r)G(r)u(r)\,dr, \qquad t\in [s,T].
\eeq
Here the operators $\{M(\cdot)\}$ and $P_T $ are linear, non-negative, bounded, selfadjoint  in $H$, the operators $\{N(\cdot)\}$ are linear, positive, bounded, selfadjoint in $U$, $x$ is an element of $H$,  the operators $\{A(r)\}_{r\in[0,T]}$ are generators of analytic semigroups $\{e^{tA(r)}\}_{t\ge 0}$ in $H$, for every $r\in[s,T]$ the family $\{U(r,s)\}_{0\le s\le r}$ is the evolution operator associated to $A(r)$ and the operator $G(r)$ is the ``Green map'' associated to $A(r)$, having the property that $[-A(r)]^\alpha G(r)$ is a bounded operator from $U$ into $H$ for some $\alpha \in \,]0,1 / 2[\,$. \\
The pair \eqref{Js}-\eqref{state_eq_mild} provides an abstract model for a large class of purely parabolic boundary control problems: the realization of $G(t)$ in concrete problems yields the lifting of the nonzero datum at the boundary, i.e. transforms the nonhomogeneous initial-boundary value problem into a homogeneous one by a modification of the right member of the evolution equation. Typical examples are given in \cite[Section 9]{AT1}. \\
Following and generalizing the methods employed in the autonomous case by \cite{LT1}, \cite{LT2} (see also the full description in the book \cite{LT3}), the above problem was analyzed in \cite{AT1} and \cite{AT2}, proving existence and uniqueness of the optimal pair $(\widehat{y},\widehat{u})$, and studying the associated Riccati equation, whose differential form is 
\beq \label{DRE}
\left\{\begin{array}{l} P'(t)+ A(t)^*P(t)+P(t)A(t) \\[2mm]
\qquad \ = -M(t)+P(t)A(t)G(t)N(t)^{-1}G(t)^*A(t)^*P(t), \quad t\in [s,T[\,,\\[2mm]
P(T) = P_T, \end{array}\right.
\eeq 
and whose integral form is 
\beq  \label{IRE} 
\begin{array}{l} \displaystyle P(t) = U(T,t)^*P_TU(T,t) \\[4mm]
\displaystyle \ \ + \int_t^T U(r,t)^*[M(r)-P(r)A(r)G(r) N(r)^{-1}G(r)^*A(r)^*P(r)]U(r,t)\,dr, \quad t\in [s,T]\,.\end{array}
\eeq
In particular, existence and representation of a classical solution $P(\cdot)$ of equation \eqref{IRE} was proved; its uniqueness was guaranteed only under further assumptions on the operator $P_T$. More precisely, it was shown, among other things, that:
\begin{description}
\item[(i)] the constructed solution $P(\cdot)$ of the Riccati equation \eqref{IRE} is classical, i.e. $P(\cdot)$ is continuously differentiable as an ${\cal L}(H)$-valued function and satisfies \eqref{DRE} in the sense of ${\cal L}(H)$, provided the operator $A(t)^*P(t)+P(t)A(t)$ is replaced by its bounded extension $\Lambda(t)P(t)$ (see \cite[Section 7]{AT1} for details);
\item[(ii)] the optimal pair enjoys some weighted H\"older continuity properties;
\item[(iii)] the final datum $P_T$ belongs to the largest possible class, according to the counterexample in \cite{F1}: namely, the composition $P_T^{1/2}L_{sT}$ is a closed operator in $H$, where $L_{sT}$ is defined in Hypothesis \ref{hpmain7} below (see also \eqref{LsT}). In addition the Riccati operator $P(\cdot)$ is strongly continuous in $[s,T]$, and a necessary and sufficient condition is given in order that $P(t)\to P_T$ in ${\cal L}(H)$ as $t\to T^-$;
\item[(iv)] uniqueness holds provided $P_T\in {\cal L}(H,D([-A(T)^*]^{2\beta}))$ for some $\beta\in \,\left]\frac12-\alpha,1\right[\,$. 
\end{description}
The main purpose of this paper is the proof, under the same assumptions made in \cite{AT1} and \cite{AT2}, of an optimal uniqueness result for the Riccati equation \eqref{IRE}: we drop the assumption written in (iv) above, and the solution turns out to be unique in a very large and natural class: the same where the existence of a solution was established. In addition we prove a regularity result for the optimal state $\widehat{y}$: although in the parabolic case, as it is well known, this function may be unbounded as $t\to T^-$, nevertheless we show that $t\mapsto P(t)^{1 / 2}\widehat{y}(t)$ is continuous in the whole interval $[s,T]$, a result that seems to be new even in the autonomous case of \cite{LT1}, \cite{LT2}. The precise assumptions of the paper are listed in Section \ref{mainres} below.\\
The question of uniqueness for Riccati equations is a delicate issue. We confine ourselves to the finite horizon theory: in the autonomous setting, uniqueness was proved in the parabolic case under the assumption that $[-A(T)^*]^{2\beta} P_T$, or $[-A(T)^*]^\beta P_T[-A(T)]^\beta$ as in \cite[Part IV, Chapter 2, Section 2.3, Theorem 2.2]{BDDM}, is bounded for some $\beta>\frac 12-\alpha$, by means of a suitable {\em a priori} bound for $G(\cdot)^*A(\cdot)^*Q(\cdot)$ (in our notation), where $Q$ is the difference of two solutions; see \cite{LT1}, \cite{LT2}, \cite{F2}, \cite[Theorem 1.5.3.3]{LT3}. In the nonautonomous setting, after the pioneering paper \cite{P}, the same method is used in \cite{AFT}. A similar argument works in the autonomous hyperbolic case, see \cite [Theorems 8.3.7.1 and 9.3.5.4]{LT4} or \cite{L}. In the intermediate case of the so called ``singular estimate control systems'' most papers use the same method: see \cite{L}, \cite{LT5}, where $P_T=0$, and \cite{LTu1}, \cite{LTu2}, \cite{LTu3} where the relevant parameter in the singular estimate assumption is less than $1/2$. A similar condition is required in \cite{Tu}, where the method for proving uniqueness is based on a ``fundamental identity'' which resembles our Lemma \ref{idcontr} below. Finally we mention the case of abstract systems arising from composite PDEs systems with boundary control: the existence of a solution of the differential Riccati equation was proved in \cite{ABL1}, while the proof of its uniqueness in the case $P_T=0$, using the fundamental identity, is in \cite {AB}. \\
We shortly describe our method: on one hand, the fundamental identity of Lemma \ref{idcontr} allows to prove that for any solution $Q$ of the integral Riccati equation it must be $P \ge Q$; on the other hand, an accurate analysis of the terms appearing in the three equivalent formulations in $[s,T-\eps]$ of the integral Riccati equation (see \eqref{Riccintw} and Lemma \ref{Riccequiv} below), as well as the careful evaluation of their limits as $\eps \to 0^+$, leads us to show that the opposite inequality $P \le Q$ also holds. Our method could perhaps be useful in improving the uniqueness results in some of the papers quoted above, too.
\begin{remark} {\em This paper should have been written more than 20 years ago, signed by the first author and his dear friend and colleague Brunello Terreni. But after the death of Brunello, for several years the first author felt unable to complete the work and even to think about this topic. Now, finally, it is time to end the whole story: this article is dedicated to the memory of Brunello.}
\end{remark}
\section{Assumptions and main results}\label{mainres}
First of all we list our assumptions (the same as in \cite{AT1}, \cite{AT2}). Let $H$ and $U$ be complex Hilbert spaces. We denote by $\Sigma^+(H)$ the class of linear, bounded, selfadjoint, non-negative operators from $H$ into itself. 
\begin{hypothesis}\label{hpmain1} For each $t\in [0,T]$, $A(t):D(A(t))\subseteq H \to H$ is a closed linear operator generating an analytic semigroup $\{e^{rA(t)}\}_{r\ge 0}$, such that $0\in \rho(A(t))$; in particular, there exist $M>0$ and $\vartheta\in \,\left]\frac{\pi}2, \pi\right[\,$ such that
$$\|(\lambda-A(t))^{-1}\|_{{\cal L}(H)} \le M(1+|\lambda|)^{-1} \qquad \forall \lambda\in \overline{S(\vartheta)}, \quad \forall t\in [0,T],$$
where $S(\vartheta) = \{z\in \C: |\textrm{{\em arg}}\,z|<\vartheta\}$.
\end{hypothesis}
\begin{hypothesis}\label{hpmain2} There exist $N>0$ and $\rho, \mu\in \,]0,1]$ with $\delta:=\rho+\mu-1\in \,]0,\frac12[\,$, such that
$$\begin{array}{l} \left\|A(t)[\lambda-A(t)]^{-1}\big[A(t)^{-1}-A(s)^{-1}\big]\right\|_{{\cal L}(H)} \\[2mm]
\qquad \qquad + \left\|A(t)^*[\lambda-A(t)^*]^{-1}\big[[A(t)^*]^{-1}-[A(s)^*]^{-1}\big]\right\|_{{\cal L}(H)} \\[2mm]
\qquad \qquad \le N|t-s|^\mu (1+|\lambda|)^{-\rho} \qquad \forall \lambda\in \overline{S(\vartheta)}, \quad \forall t,s\in [0,T].\end{array}$$
\end{hypothesis}
By the results of \cite{AT} and \cite {A}, the family $\{A(t)\}_{t\in [0,T]}$ generates in $H$ a strongly continuous evolution operator $U(t,s)$.
\begin{hypothesis}\label{hpmain3} The family $\{U(t,s)\}_{0\le s\le t\le T}$ is the evolution operator of $\{A(t)\}_{t\in [0,T]}$ and satisfies 
$$\begin{array}{r} \left\|[-A(t)]^\eta U(t,s)[-A(s)]^{-\gamma}\right\|_{{\cal L}(H)}+\left\|[-A(s)^*]^\eta U(t,s)^*[-A(t)^*]^{-\gamma}\right\|_{{\cal L}(H)} \\[2mm]
\le M_{\eta\gamma} [1+(t-s)^{\eta-\gamma}]\qquad \textrm{for }\ 0\le s< t\le T, \quad \eta, \gamma \in [0,1].\end{array}$$
\end{hypothesis}
\begin{hypothesis}\label{hpmain4} The number $\delta =\rho+\mu-1$ is such that $0<\delta<\frac12$ and
$$\begin{array}{l} \left\|[-A(t)]^\eta U(t,s)[-A(s)]^{-\gamma}-[-A(\tau)]^\eta U(\tau,s)[-A(s)]^{-\gamma}\right\|_{{\cal L}(H)} \\[2mm]
\qquad \qquad \le N_{\eta\gamma} (t-\tau)^\delta [1+(\tau-s)^{\gamma-\eta-\delta}] \qquad \textrm{for }\ 0\le s< \tau\le t \le T, \quad \eta, \gamma \in [0,1],\\[4mm]
\left\|[-A(\sigma)^*]^\eta U(t,\sigma)^*[-A(t)^*]^{-\gamma}-[-A(s)^*]^\eta U(t,s)^*[-A(t)^*]^{-\gamma}\right\|_{{\cal L}(H)} \\[2mm]
\qquad \qquad \le N_{\eta\gamma} (\sigma-s)^\delta [1+(t-\sigma)^{\gamma-\eta-\delta}] \qquad \textrm{for }\ 0\le s\le\sigma< t \le T, \quad \eta, \gamma \in [0,1],
\end{array}$$ 
all operators being strongly continuous with respect to $t,\tau, \sigma, s$. 
\end{hypothesis} 
The properties of Hypothesis \ref{hpmain3} were proved in \cite{A} and follow by the results of \cite{AT} and \cite{AT0}. The statement of Hypothesis \ref{hpmain4} was never proved explicitly, but it follows essentially by estimates contained in \cite{A} and \cite {AT0}.
\begin{hypothesis}\label{hpmain5} $G(t)\in {\cal L}(U,H)$ for every $t\in [0,T]$, and there exists $\alpha\in \,\left]\delta,\frac12\right[\,$ such that
$$t \mapsto [-A(t)]^\alpha G(t)\in C^\delta([0,T],{\cal L}(U,H)).$$
\end{hypothesis}
The verification of the abstract Hypotheses \ref{hpmain3} and \ref{hpmain5} in several concrete initial boundary value problems for parabolic systems was made in \cite{AFT}. \\
By Hypotheses \ref{hpmain4} and \ref{hpmain5} we can give a precise meaning to the expression $U(q,s)A(s)G(s)$, $0\le s \le t\le T$, often appearing in the sequel: namely
\beq\label{UAG} 
U(t,s)A(s)G(s) = -[[-A(s)^*]^{1-\alpha}U(t,s)^*]^*[-A(s)]^\alpha G(s),
\eeq
and this operator is estimated by
\beq\label{stimaUAG}
\\|[[-A(s)^*]^{1-\alpha}U(t,s)^*]^*[-A(s)]^\alpha G(s)\|_{{\cal l}(U,H)} \le c(t-s)^{\alpha-1}.
\eeq
\begin{hypothesis}\label{hpmain6} $M(\cdot)\in C^\delta([0,T],\Sigma^+(H))$,  $N(\cdot)\in C^\delta([0,T],\Sigma^+(U))$, and there exists $\nu>0$ such that 
$$\langle N(t)u,u\rangle_U \ge \nu\|u\|_U^2 \qquad \forall t\in [0,T], \quad \forall u\in U.$$
\end{hypothesis}
This assumption is essential in the proof of the existence of a unique optimal control for the functional \eqref{Js} under the constraint \eqref{state_eq_mild}. \\
Before stating Hypothesis \ref{hpmain7}, it is useful to introduce some notations. We rewrite equation \eqref{state_eq_mild} as
\beq\label{state_Ls} 
y(t) = U(t,s)x + [L_su](t)\, \qquad t\in [s,T].
\eeq
where, of course (recalling \eqref{UAG}),
\beq\label{Ls} 
L_su(t) = -\int_s^t U(t,r)A(r)G(r)u(r)\,dr, \qquad t\in [s,T[\,, \qquad u\in L^2(s,T;U).
\eeq
The properties of the operator $L_s$ and of its adjoint $L_s^*$ are listed in \cite[Lemma 4.4]{AT1}. Wec also introduce the operator $L_{sT}$, $0\le s<T$, as
\beq\label{LsT} 
\left\{\begin{array}{l}D(L_{sT})=\displaystyle\left\{u\in L^2(s,T;U) :\ \int_s^T A(T)^{-1}U(T,r)A(r)G(r) u(r)dr\in D(A(T))\right\} \\[3mm]
\displaystyle L_{sT}u = -A(T)\int_s^T A(T)^{-1}U(T,r)A(r)G(r) u(r)dr \qquad \forall u\in D(L_{sT}).\end{array}\right.
\eeq
Clearly it holds $D(L_{sT})=D(L_{0T})$ for every $s\in \,]0,T[\,$, since  
$$\begin{array}{r} \displaystyle \int_s^T A(T)^{-1}U(T,r)A(r)G(r) u(r)dr -\int_0^T A(T)^{-1}U(T,r)A(r)G(r) u(r)dr \\[4mm]
\displaystyle = - A(T)^{-1}\int_0^s U(T,r)A(r)G(r) u(r)dr\in D(A(T)),
\end{array}$$
and for every $u\in D(L_{sT})$ 
\beq\label{Ls0T}
L_{0T}u = L_{sT}u -\int_0^s U(T,r)A(r)G(r)u(r)\,dr = L_{sT}u + U(T,s)L_0u(s).
\eeq
Let us finally state the crucial Hypothesis \ref{hpmain7}
\begin{hypothesis}\label{hpmain7} $P_T\in \Sigma^+(H)$, and the linear operator 
$$P_T^{1/2}L_{0T}:D(L_{0T})\subseteq L^2(0,T;U) \to H$$ is closed.
\end{hypothesis}
We cannot drop this assumption, due to the counterexample in \cite{F1} (see also \cite[Section 1.7]{LT3}).  
\begin{remark}\label{rem_ass} {\em As said before, the above assumptions are precisely the same as in the papers \cite{AT1} and \cite{AT2}.}\quad \qed
\end{remark}
As a consequence of Hypothesis \ref{hpmain7}, since the admissible controls are precisely the elements of $D(L_{sT})$, we can rewrite more exactly the cost functional as
\beq \label{Js_preciso}
J_s(u) \hspace{-1mm} = \hspace{-1mm}\left\{ \begin{array}{l} \displaystyle \int_s^T \|M(r)^{1 / 2}y(r)\|_H^2\,dr + \int_s^T \|N(r)^{1 / 2}u(r)\|_U^2\,dr + \|P_T^{1/2}y(T)\|_H^2 \\[2mm]
\quad\qquad\qquad\qquad\qquad\qquad\qquad \qquad\ \ \textrm{ if } u\in D(L_{sT}) \\[2mm]
+ \infty \ \ \ \quad\qquad\qquad\qquad\qquad\qquad\qquad \textrm{ if } u\in L^2(s,T;U)\setminus D(L_{sT}),\end{array}\right.
\eeq
where $P_T^{1/2}y(T)$ means $P_T^{1/2}U(T,s)x+P_T^{1/2}L_{sT}u$ for every $u\in D(L_{sT})$.\\
For further use, we also recall the definition of $L_{sT}^*$ (see  \cite[formula (2.8)]{AT1},
\beq \label{LsT*}\left\{ \begin{array}{l} 
D(L_{sT}^*) =\{y\in H: \ G(\cdot)^*A(\cdot)^*U(T,\cdot)^*y\in L^2(s,T;U)\}\\[2mm]
L_{sT}^*y = -G(\cdot)^*A(\cdot)^*U(T,\cdot)^*y; \end{array}\right.
\eeq
note that, by definition, $D(L_{tT}^*)\subseteq D(L_{sT}^*)$ for every $s\in [0,t]$, and that $L_{tT}^*y=L_{sT}^*y|_{[t,T[}$ for every $y\in D(L_{sT}$.\\ 
We now state our main results.  
\begin{theorem}\label{unic} The Riccati equation \eqref{IRE} has a unique solution $P$ within the class 
\beq \label{classeQ}
\begin{array}{l}{\cal Q}=\Big\{Q\in L^\infty(0,T;\Sigma^+(H))\cap C([0,T[\,,\Sigma^+(H)):\\[2mm] 
\qquad \qquad \qquad\displaystyle \lim_{t\to T^-} \|Q(t)x-P_Tx\|_H=0\quad \forall x\in H,\\[3mm]
\qquad \qquad \qquad\displaystyle \left\| \left[-A(t)^*\right]^{1-\alpha}Q(t)\right\|_{{\cal L}(H)} \le c(T-t)^{\alpha-1} \quad\forall t\in [0,T[\Big\}. \end{array} 
\eeq
\end{theorem}
The proof is postponed in Section \ref{proof_unic}.
\begin{theorem}\label{reg_Py} The optimal state $\widehat{y}$ of problem \eqref{Js_preciso}-\eqref{state_Ls} satisfies 
$$t\mapsto P(t)^{1/2}\widehat{y}(t) \in C([0,T],H).$$
\end{theorem}
This theorem will follow as a simple consequence of Theorem \ref{unic}, at the end of Section \ref{proof_unic}. 
\section{Auxiliary results} 
Our proof needs some auxiliary facts. In the next statements we consider a fixed operator $Q$, belonging to the class ${\cal Q}$ introduced in \eqref{classeQ}, which satisfies the differential Riccati equation \eqref{DRE}. As we do not know yet whether the singularity of the quadratic term is integrable at $T$ or not, we rewrite \eqref{DRE} as an integral equation in the smaller interval $[s,T-\eps]$, where $0\le s <T-\eps<T$. We write it in its weak form, valid for all $x,y\in H$:
\beq \label{Riccintw}
\begin{array}{lll}
\langle Q(s)x,y\rangle_H = \langle Q(T-\eps)U(T-\eps,s)x,U(T-\eps,s)y\rangle_H  \\[2mm]
\displaystyle \qquad \qquad \quad \ +\int_s^{T-\eps} \langle M(r)U(r,s)x,U(r,s)y\rangle_H\,dr  \\[4mm]
\displaystyle \qquad \qquad \quad \ -\int_s^{T-\eps} \langle N(r)^{-1}G(r)^*A(r)^*Q(r)U(r,s)x,G(r)^*A(r)^*Q(r)U(r,s)y\rangle_U \,dr,
\end{array}
\eeq 
We recall explicitly that in \eqref{Riccintw} the operators $Q(\cdot)$, $U(\cdot,\cdot)$, $M(\cdot)$, $N(\cdot)^{-1}$ are uniformly bounded in $[0,T]$ by some constant $K$, while the operator $G(\cdot)^*A(\cdot)^*Q(\cdot)$, taking into account \eqref{UAG} and \eqref{stimaUAG}, is well defined and uniformly bounded in $[s,T-\eps]$ by some constant $C_\eps\,$, since its singularity is concentrated at $T$. \\
We oserve now that we may let $\eps\to 0^+$ in equation \eqref{Riccintw}: indeed all terms converge, but (possibly) the last. Hence, by difference, the last one converges, too, and when $x=y$ we deduce 
$$N(\cdot)^{-1 / 2}G(\cdot)^*A(\cdot)^*Q(\cdot)U(\cdot,s)x \in L^2(s,T;U).$$
Thus, we obtain from \eqref{Riccintw}, for all $x,y\in H$, as $\eps\to 0^+$:
\beq \label{Riccintw2}
\begin{array}{lll}
\langle \displaystyle Q(s)x,y\rangle_H = \langle P_TU(T,s)x,U(T,s)y\rangle_H + \int_s^{T} \langle M(r)U(r,s)x,U(r,s)y\rangle_H\,dr  \\[4mm]
\displaystyle \qquad \qquad \quad \ -\int_s^{T} \langle N(r)^{-1}G(r)^*A(r)^*Q(r)U(r,s)x,G(r)^*A(r)^*Q(r)U(r,s)y\rangle_U \,dr.
\end{array}
\eeq 
Our first lemma concerns a basic identity which is well known in classical control theory.
\begin{lemma}[Fundamental identity]\label{idcontr} Let ${\cal Q}$ be given by (\ref{classeQ}), and let $Q\in {\cal Q}$ be a solution of \eqref{Riccintw}. In addition, fix $s\in [0,T[\,$, $x\in H$, a control $u\in L^2(s,T;U)$ and the corresponding state $y\in L^2(s,T;H)$, given by \eqref{state_Ls}. Then the following identity holds for every $\eps \in \,]0,T-s]$:
\beq \label{idbase} 
\begin{array}{l} \langle Q(T-\eps)y(T-\eps),y(T-\eps)\rangle_H - \langle Q(s)x,x\rangle_H  \\[2mm]
\quad \qquad \qquad \quad \ \displaystyle =-\int_s^{T-\eps} \|M(r)^{1 / 2}y(r)\|_H^2\,dr -\int_s^{T-\eps} \|N(r)^{1 / 2}u(r)\|_U^2\,dr \\[4mm]
\quad \qquad \qquad \quad \ \displaystyle +\int_s^{T-\eps} \|N(r)^{1 / 2}u(r)-N(r)^{-1 / 2}G(r)^*A(r)^*Q(r)y(r)\|_U^2 \,dr.\end{array}
\eeq 
\end{lemma} 
The proof is in Appendix \ref{A}.\\[2mm]
Our second lemma concerns the so called ``closed loop equation''.
\begin{lemma}\label{closedloop} Let $Q\in {\cal Q}$, with ${\cal Q}$ given by \eqref{classeQ}. For fixed $s\in [0,T[\,$, the closed loop equation 
\beq \label{eqclosedloop}
\Phi(t,s)=U(t,s)-\int_s^t U(t,r)A(r)G(r)N(r)^{-1}G(r)^*A(r)^*Q(r)\Phi(r,s)\,dr, \quad t\in [s,T[\,,
\eeq
has a unique solution $\Phi(\cdot,s)$, belonging to $C([s,T-\eps],{\cal L}(H))$ for every $\eps \in \,]0,T-s[\,$.\end{lemma}
\begin{proof} Fix $\eps\in \,]0,T-s[\,$. We introduce the integral operator
$$[K_s g](t):= \int_s^t U(t,r)A(r)G(r)N(r)^{-1}G(r)^*A(r)^*Q(r)g(r)\,dr, \ \ t\in [s,T-\eps],\ \ g\in L^2(s,T;U),$$
whose kernel $K(t,r)$, given by
$$K(t,r):= U(t,r)A(r)G(r)N(r)^{-1}G(r)^*A(r)^*Q(r), \quad 0\le r< t< T,$$
is continuous in the region $\{(t,r): 0\le r<t<T\}$ with values in ${\cal L}(H)$ and satisfies the estimate (see \cite[formula (6.6)]{AT1})
$$\|K(t,r)\|_{{\cal L}(H)}\le c(t-r)^{\alpha-1}(T-r)^{\alpha-1}, \qquad 0 \le r < t < T,$$
so that, in particular,
\beq \label{stimanucleo}
\|K(t,r)\|_{{\cal L}(H)} \le c_\eps (t-r)^{\alpha-1}, \quad 0\le r < t \le T-\eps.
\eeq
It is shown in \cite{AT2}, among other things, that $(1 + K_s)^{-1}$ is well defined and belongs to the space ${\cal L}(C([s,T-\eps],{\cal L}(H)))$. Thus if we set $\Phi(t,s)=[(1 + K_s)^{-1}(U(\cdot,s)](t)$, we immediately obtain that \eqref{eqclosedloop} holds in $[s,T-\eps]$, with arbitrary $\eps\in \,]0,T-s[\,$, i.e. in $[s,T[\,$. 
\end{proof}
We remark that, by uniqueness of the solution of \eqref{state_eq_strong}, 
\beq \label{transit}
\Phi(t,s)=\Phi(t,q)\Phi(q,s),  \qquad 0\le s \le q \le t <T.
\eeq
The next lemma is basic: it expresses the integral Riccati equation \eqref{IRE} in two equivalent forms, from which we will deduce all the relevant informations for our proof.
\begin{lemma} \label{Riccequiv} Let $\Phi(t,s)$ be defined by Lemma \ref{closedloop}. The Riccati integral equation \eqref{Riccintw} is equivalent to both equations below, which hold for all $x,y\in H$ and $0\le s\le T-\eps<T$:
\beq \label{Riccint1}
\begin{array}{l}
\langle Q(s)x,y\rangle_H  \\
\quad =\displaystyle\langle Q(T-\eps)\Phi(T-\eps,s)x,U(T-\eps,s)y\rangle_H + 
\int_s^{T-\eps} \langle M(r)\Phi(r,s)x,U(r,s)y\rangle_H\,dr; \end{array}
\eeq
\beq \label{Riccint2}
\begin{array}{l}
\hspace{-4mm}\langle Q(s)x,y\rangle_H = \\
\hspace{-4mm}\ \ =\displaystyle\langle Q(T-\eps)\Phi(T-\eps,s)x,\Phi(T-\eps,s)y\rangle_H +\int_s^{T-\eps} \langle M(r)\Phi(r,s)x,\Phi(r,s)y\rangle_H\,dr + \\[2mm]
\hspace{-4mm} \ \ \displaystyle \ \ \quad \qquad+\int_s^{T-\eps} \langle N(r)^{-1}G(r)^*A(r)^*Q(r)\Phi(r,s)x,G(r)^*A(r)^*Q(r)\Phi(r,s)y\rangle_U \,dr.
\end{array}
\eeq
\end{lemma}
The proof of Lemma \ref{Riccequiv} is in Appendix \ref{B}.
Equation \eqref{Riccint2} is specially useful for our purposes: indeed, when $y=x$ it can be written as 
\beq\label{Riccint2x}
\begin{array}{l}\displaystyle \|Q(s)^{1 / 2}x\|_H^2 = \|Q(T-\eps)^{1 / 2}\Phi(T-\eps,s)x\|_H^2 +\int_s^{T-\eps} \|M(r)^{1 / 2}\Phi(r,s)x\|_H^2 \,dr\\[2mm]
\quad \quad \quad \ \ \displaystyle \qquad +\int_s^{T-\eps} \|N(r)^{-1 / 2}G(r)^*A(r)^*Q(r)\Phi(r,s)x\|_U^2 \,dr, \qquad x\in H.\end{array}
\eeq
Since the integrals are bounded and  monotonically increasing as $\eps \to 0^+$, the first term in the right member is bounded and decreasing. Thus, if $Q\in {\cal Q}$ is a solution of \eqref{Riccintw}, then by \eqref{Riccint2x} we can define
\beq \label{b(s)}
\langle B(s)x,x\rangle_H :=\lim_{\eps \to 0^+} \langle Q(T-\eps)\Phi(T-\eps,s)x,\Phi(T-\eps,s)x\rangle_H\,,\quad x\in H. 
\eeq
We have $B(s)\in \Sigma^+(H)$ and equation \eqref{Riccint2x} becomes, as $\eps\to 0^+$,
\beq \label{Riccint3}
\begin{array}{lll}
\displaystyle \langle Q(s)x,x\rangle_H = \langle B(s)x,x\rangle_H + \int_s^T \|M(r)^{1 / 2}\Phi(r,s)x\|_H^2\,dr  \\[2mm]
\qquad \qquad \quad\ \displaystyle +\int_s^T \|N(r)^{-1 / 2}G(r)^*A(r)^*Q(r)\Phi(r,s)x\|_U^2 \,dr, \quad x\in H.
\end{array}
\eeq
Then we can pass to the limit as $\eps\to 0^+$ in \eqref{Riccint2}, too; using polarization in \eqref{b(s)}, we get
\beq\label{b(s)xy}
\langle B(s)x,y\rangle_H=\lim_{\eps \to 0^+} \langle Q(T-\eps)\Phi(T-\eps,s)x,\Phi(T-\eps,s)y\rangle_H\,,\quad x,y\in H,
\eeq
and
\beq \label{Riccint3bis}
\begin{array}{lll}
\displaystyle \langle Q(s)x,y\rangle_H = \langle B(s)x,y\rangle_H + \int_s^T \langle M(r)\Phi(r,s)x,\Phi(r,s)y\rangle_H\,dr \\[2mm]
\quad \displaystyle +\int_s^T \langle N(r)^{-1}G(r)^*A(r)^*Q(r)\Phi(r,s)x,G(r)^*A(r)^*Q(r)\Phi(r,s)y\rangle_U \,dr, \quad x,y\in H.
\end{array}
\eeq
We now recall some properties of the optimal pair $(\widehat{y}, \widehat{u})$ with initial point $x\in H$ and initial time $s$, whose existence was proved in \cite{AT1} and \cite{AT2}. By \cite[Proposition 5.4]{AT1}, it holds
\beq\label{P(s)xx}
\langle P(s)x,x\rangle_H = \hspace{-1mm }J_s(\widehat{u}) = \hspace{-1mm }\int_s^T \|M(r)^{1 / 2}\widehat{y}(r)\|^2_H \, dr + \hspace{-1mm }\int_s^T \|N(r)^{1 / 2}\widehat{u}(r)\|_U^2 \,dr + \|P_T^{1/2}\widehat{y}(T)\|_H^2\,;
\eeq
here $P(\cdot)$ is the classical solution, constructed in \cite{AT1}, of the differential Riccati equation \eqref{DRE}. The optimal control $\widehat{u}$ is given in feedback form by
\beq \label{optcontr_feedback}
\widehat{u}(t;s,x) = N(t)^{-1}G(t)^*A(t)^*P(t)\widehat{y}(t),\quad t\in [s,T[\,.
\eeq
By optimality, $\widehat{u}=\widehat{u}(\cdot;s,x)$ belongs to $D(L_{sT})$. The optimal state $\widehat{y}=\widehat{y}(\cdot;s,x)$ is
\beq \label{opt_state} 
\widehat{y}(t) = \widehat{\Phi}(t,s)x=U(t,s)x +[L_s\widehat{u}](t),\quad t\in [s,T[\,.
\eeq
Now, let $Q(\cdot)$ be any solution of the Riccati equation \eqref{Riccintw}, belonging to the class ${\cal Q}$ defined by (\ref{classeQ}). In order to prove that $P=Q$, we introduce the pair  $(\bar{y},\bar{u})$ where, similarly to \eqref{optcontr_feedback},  $\bar{u}$ is defined in feedback form in terms of $\bar{y}$ and $Q$: namely, for fixed $s\in [0,T[\,$, with $\Phi(t,s)$ defined by \eqref{eqclosedloop}, we set
\beq \label{controlbar} 
\bar{u}(t)=\bar{u}(t;s,x) = N(t)^{-1}G(t)^*A(t)^*Q(t)\Phi(t,s)x,\quad s\le t < T,
\eeq
and
\beq \label{statebar}
\bar{y}(t)=\bar{y}(t;s,x)= \Phi(t,s)x= U(t,s)x+[L_s \bar{u}(\cdot;s,x)](t),\quad s\le t < T.
\eeq
By \eqref{Riccint3} we have $\bar{u}\in L^2(s,T;U)$ and hence, by \cite[Lemma 4.4(i)]{AT1}, $\bar{y}\in L^{\frac2{1-2\alpha}}(s,T;H)$; we note explicitly that, by \eqref{transit}, 
\beq\label{transitu} 
\bar{u}(r;s,x) = \bar{u}(r;t,\Phi(t,s)x)\qquad \forall r\in [t,T[\,, \quad \forall t\in [s,T[\,.
\eeq
Since we do not know whether or not $\bar{u}\in D(L_{sT})$, we define suitable  approximations of $\bar{u}$ and $\bar{y}$:
\beq \label{u_eps_y_eps} u_\eps = u_\eps(\cdot;s,x)=\left\{ \begin{array}{ll} \bar{u}(\cdot;s,x) & \textrm{in } [s,T-\eps] \\[2mm]  0 & \textrm{in } \,]T-\eps,T],\end{array}\right. \qquad y_\eps = U(\cdot,s)x+ L_s u_\eps;
\eeq
then in particular 
\beq\label{y_eps} y_\eps(t)=y_\eps(t;s,x) = \left\{ \begin{array}{ll} \Phi(t,s)x & \textrm{if } t\in [s,T-\eps[ \\[2mm]
U(t,T-\eps)\Phi(T-\eps,s)x & \textrm{if } t\in [T-\eps,T].\end{array}\right.
\eeq
Note that $u_\eps\in D(L_{sT})$, with $L_{sT}u_\eps =U(T,T-\eps)[L_s\bar{u}](T-\eps)$; moreover, $u_\eps \to \bar{u}$ in $L^2(s,T;U)$ and, using \cite[Lemma 4.4(i)]{AT1}, 
\beq\label{regPhi} y_\eps \to \bar{y} \ \textrm{ in } \ L^{\frac2{1-2\alpha}}(s,T;H) \ \textrm{ as } \ \eps \to 0^+.
\eeq 
\section{Proof of the main results} \label{proof_unic}
Let $Q(\cdot)$ be any solution of the Riccati equation, belonging to the class ${\cal Q}$ defined by (\ref{classeQ}). We will prove the two inequalities
\beq\label{tesi}
\begin{array}{c}{\textrm{\bf (a)}} \ \langle P(s)x,x\rangle_H \ge \langle Q(s)x,x\rangle _H \quad \forall x\in H, \\[2mm]
{\textrm{\bf (b)}}\ \langle P(s)x,x\rangle_H \le \langle Q(s)x,x\rangle _H\quad \forall x\in H.\end{array}
\eeq
Theorem \ref{unic} will then follow by the arbitrariness of $s\in [0,T[\,$ and a simple polarization argument.
\subsection{Proof of Theorem \ref{unic}(a)} \label{PgeQ}  
We start with the following proposition:
\begin{proposition}\label{PleQ1} Let $(\widehat{u},\widehat{y})\in D(L_{sT})\times L^2(s,T;H)$ be the optimal pair, given by \eqref{statebar}-\eqref{controlbar}, and set for every $\eps\in \,]0,T-s[\,$ 
$$\widehat{u}_\eps = \left\{ \begin{array}{ll} \widehat{u} & \quad \textrm{in}\ [s,T-\eps] \\[2mm]  0 & \quad  \textrm{in}\ ]T-\eps,T],\end{array}\right. \quad \widehat{y}_\eps = U(\cdot,s)x+L_s \widehat{u}_\eps;$$
then we have $\widehat{y}_\eps\in C([s,T],H)$ and
\beq\label{optugu}\lim_{\eps\to 0^+}\langle P_T\widehat{y}_\eps(T),\widehat{y}_\eps(T)\rangle_H = \langle P_T\widehat{y}(T),\widehat{y}(T)\rangle_H = \lim_{\eps\to 0^+}\langle P(T-\eps)\widehat{y}(T-\eps),\widehat{y}(T-\eps)\rangle_H\,.
\eeq
\end{proposition}
\begin{proof}
It is well known that $\widehat{y}\in C([s,T[\,,H)$, so that $\widehat{y}_\eps\in C([s,T-\eps],H)$. On the other hand
$$\widehat{y}_\eps(t)= U(t,s)x+[L_s\widehat{u}_\eps](t)=U(t,s)x+U(t,T-\eps)[L_s\widehat{u}](T-\eps) \qquad \forall t\in [T-\eps,T];$$
this implies $\widehat{y}_\eps\in C([s,T],H)$. Next, we note that, by the definition of $\widehat{u}_\eps$\,,
$$L_{sT}(\widehat{u}_\varepsilon-\widehat{u}) = A(T) \int_{T-\eps}^T A(T)^{-1}U(T,r)A(r)G(r)\widehat{u}(r)\,dr = -L_{T-\varepsilon\,T} \widehat{u};$$
Now we have obviously
$$\widehat{u}_\varepsilon\to \widehat{u} \quad \textrm{in } L^2(s,T;U) \quad \textrm{as } \eps\to 0^+;$$
moreover, by \cite[Proposition 4.2(ii)]{AT1},
$$P_T^{1/2}L_{sT}(\widehat{u}_\varepsilon - \widehat{u}) = - P_T^{1/2}L_{T-\varepsilon\,T} \widehat{u} \to 0 \quad \textrm{in } H \quad \textrm{as } \eps\to 0^+.$$
Hence, as $ \eps\to 0^+$, 
$$P_T^{1/2}\widehat{y}_\varepsilon(T)=P_T^{1/2}U(T,s)x+P_T^{1/2}L_{sT}\widehat{u}_\varepsilon \to P_T^{1/2}U(T,s)x+P_T^{1/2}L_{sT}\widehat{u}=P_T^{1/2}\widehat{y}(T) \quad \textrm{in } H,$$
which implies the first equality in \eqref{optugu}.\\
Next, choosing in Lemma \ref{idcontr} $Q=P$, $u=\widehat{u}$, $y=\widehat{y}$, we obtain by \eqref{idbase} and \eqref{optcontr_feedback}:
$$\begin{array}{l}
\langle P(T-\eps)\widehat{y}(T-\eps),\widehat{y}(T-\eps)\rangle_H - \langle P(s)x,x\rangle_H \\[2mm]
\quad \qquad \qquad \quad \ \displaystyle =-\int_s^{T-\eps} \|M(r)^{1 / 2}\widehat{y}(r)\|_H^2\,dr -\int_s^{T-\eps} \|N(r)^{1 / 2}\widehat{u}(r)\|_U^2\,dr + 0,\end{array}$$
and letting $\eps \to 0^+$ we get, using \eqref{P(s)xx},
\beyy
\lefteqn{\lim_{\eps\to 0^+}\langle P(T-\eps)\widehat{y}(T-\eps),\widehat{y}(T-\eps)\rangle_H}\\
& & =\langle P(s)x,x\rangle_H -\int_s^T \|M(r)^{1 / 2}\widehat{y}(r)\|_H^2\,dr -\int_s^T \|N(r)^{1 / 2}\widehat{u}(r)\|_U^2\,dr \\
& & = J_s(\widehat{u}) -\int_s^T \|M(r)^{1 / 2}\widehat{y}(r)\|_H^2\,dr -\int_s^T \|N(r)^{1 / 2}\widehat{u}(r)\|_U^2\,dr = \langle P_T\widehat{y}(T),\widehat{y}(T)\rangle_H\,,
\eeyy
which gives the second equality in \eqref{optugu}. 
\end{proof}
We now complete the proof of {\bf (a)}. For fixed $\eps\in\,]0,T-s[\,$,  replace in \eqref{idbase} $\eps$ by $\delta$, with $0<\delta <\eps$, and choose 
$u=\widehat{u}_\eps$, $y=\widehat{y}_\eps$. Then we find 
\beyy
\langle Q(s)x,x\rangle_H & = & \langle Q(T-\delta)\widehat{y}_\eps(T-\delta), \widehat{y}_\eps(T-\delta)\rangle_H \\[2mm]
& & +\int_s^{T-\delta}\langle M(r)\widehat{y}_\eps(r),\widehat{y}_\eps(r)\rangle_H \,dr + \int_s^{T-\delta}\langle N(r)\widehat{u}_\eps(r),\widehat{u}_\eps(r)\rangle_U \,dr \\
& & -\int_s^{T-\delta}\|N(r)^{1 / 2}\widehat{u}_\eps(r) -N(r)^{-1 / 2} G(r)^*A(r)^*Q(r)\widehat{y}_\eps(r)\|_U^2\,dr \\[2mm]
& \le & \langle Q(T-\delta)\widehat{y}_\eps(T-\delta), \widehat{y}_\eps(T-\delta)\rangle_H \\[2mm]
& & +\int_s^{T-\delta}\langle M(r)\widehat{y}_\eps(r),\widehat{y}_\eps(r)\rangle_H \,dr + \int_s^{T-\delta}\langle N(r)\widehat{u}_\eps(r),\widehat{u}_\eps(r)\rangle_U \,dr.
\eeyy
As $\delta\to 0^+$, since we know that $\widehat{y}_\eps\in C([s,T],H)$ we deduce
\beyy
\langle Q(s)x,x\rangle_H & \le & \langle P_T\widehat{y}_\eps(T), \widehat{y}_\eps(T)\rangle_H \\
& & +\int_s^{T}\langle M(r)\widehat{y}_\eps(r),\widehat{y}_\eps(r)\rangle_H \,dr + \int_s^{T}\langle N(r)\widehat{u}_\eps(r),\widehat{u}_\eps(r)\rangle_U \,dr.
\eeyy
Finally, we let $\eps \to 0^+$: noting that $\widehat{y}_\eps\to \widehat{y}$ in $L^2(s,T;H)$ and $\widehat{u}_\eps\to \widehat{u}$ in $L^2(s,T;U)$, and using the first equality in \eqref{optugu}, we get
\beyy
\langle Q(s)x,x\rangle_H & \le & \langle P_T\widehat{y}(T),\widehat{y}(T)\rangle_H + \int_s^{T}\langle M(r)\widehat{y}(r),\widehat{y}(r)\rangle_H \,dr + \int_s^{T}\langle N(r)\widehat{u}(r),\widehat{u}(r)\rangle_U \,dr \\
& = & J_s(\widehat{u}) = \langle P(s)x,x\rangle_H\,.
\eeyy
This proves {\bf (a)} in \eqref{tesi}.
\subsection{Proof of Theorem \ref{unic}(b)} \label{PleQ}  
This proof is longer. \\
First, we list some statements where certain terms of equations \eqref{Riccintw2}, \eqref{Riccint1} and \eqref{Riccint2} are analyzed. 
The first one is easy and concerns two integral terms in \eqref{Riccintw2} (with $y=x$) and \eqref{Riccint3}.
\begin{lemma}\label{Mu_Mphi} We have 
$$
\lim_{\eps\to 0^+} \int_{T-\eps}^T \|M(r)^{1 / 2}U(r,T-\eps)x\|_H^2\,dr =\lim_{\eps\to 0^+} \int_{T-\eps}^T \|M(r)^{1 / 2}\Phi(r,T-\eps)x\|_H^2\,dr = 0.
$$
\end{lemma} 
\begin{proof}
The first limit is trivial, since $M(\cdot)^{1 / 2}U(\cdot,T-\eps)$ is uniformly bounded in ${\cal L}(H)$, independently on $\eps$. Concerning the second one, we have, using H\"older inequality and \eqref{regPhi}: 
$$\int_{T-\eps}^T \|M(r)^{1 / 2}\Phi(r,T-\eps)x\|_H^2\,dr \le K_1\,\eps^{2\alpha} \|\bar{y}(\cdot;T-\eps,x)\|_{L^{\frac2{1-2\alpha}}(T-\eps,T;H)}^2 \le K_2\,\eps^{2\alpha}\|x\|_H^2$$
for some absolute constants $K_1,K_2>0$. The result follows.  
\end{proof}
The second lemma analyzes the last integral term of \eqref{Riccintw2} with $y=x$. 
\begin{lemma}\label{stimaU} We have
$$\lim_{\eps\to 0^+} \int_{T-\eps}^T \|N(r)^{-1 / 2}G(r)^* A(r)^*Q(r)U(r,T-\eps)x\|_U^2\,dr = 0.$$
\end{lemma}
\begin{proof} By \eqref{Riccintw2} with $y=x$ and $s=T-\eps$ we have
\beyy
\lefteqn{\int_{T-\eps}^T \|N(r)^{-1 / 2}G(r)^* A(r)^*Q(r)U(r,T-\eps)x\|_U^2\,dr } \\
& & = \|P_T^{1/2}U(T,T-\eps)x\|_H^2 + \int_{T-\eps}^T \|M(r)^{1 / 2} U(r,T-\eps)x\|_H^2\,dr - \|Q(T-\eps)^{1 / 2}x\|_H^2 \,.
\eeyy
As $\eps\to 0^+$ the result follows, since in the right-hand side the integral term goes to $0$ by Lemma \ref{Mu_Mphi}, while both $P_T^{1/2}U(T,T-\eps)x$ and $Q(T-\eps)^{1 / 2}x$ converge in $H$ to $P_T^{1/2}x$ as $\eps\to 0^+$.  
\end{proof}
The next result concerns the last integral term of \eqref{Riccint3}.
\begin{lemma}\label{stimaPhi} We have
$$\lim_{\eps\to 0^+} \int_{T-\eps}^T \|N(r)^{-1 / 2}G(r)^* A(r)^*Q(r)\Phi(r,T-\eps)x\|_U^2\,dr =0.$$
\end{lemma}
\begin{proof} 
With $y=x$, we sum equation \eqref{Riccintw2},  minus twice equation \eqref{Riccint1}, plus equation \eqref{Riccint2}. The result is
\beq \label{a+c-2b} 
\begin{array}{ll} 
0 = & \displaystyle \|Q(T-\eps)^{1 / 2}U(T-\eps,s)x-Q(T-\eps)^{1 / 2}\Phi(T-\eps,s)x\|_H^2 \\[2mm]
& \displaystyle \qquad + \int_s^{T-\eps} \|M(r)^{1 / 2}U(r,s)x-M(r)^{1 / 2}\Phi(r,s)x\|_H^2\,dr \\[4mm]
& \displaystyle \qquad - \int_s^{T-\eps} \|N(r)^{-1 / 2}G(r)^* A(r)^*Q(r)U(r,s)x\|_U^2\,dr \\[4mm]
& \displaystyle \qquad + \int_s^{T-\eps} \|N(r)^{-1 / 2}G(r)^* A(r)^*Q(r)\Phi(r,s)x\|_U^2\,dr.
\end{array}
\eeq
Replace $\eps$ by $\delta$. It follows that
$$\int_s^{T-\delta}\hspace{-2mm}\|N(r)^{-1 / 2}G(r)^* A(r)^*Q(r)\Phi(r,s)x\|_U^2\,dr \le \int_s^{T-\delta} \hspace{-2mm}\|N(r)^{-1 / 2}G(r)^* A(r)^*Q(r)U(r,s)x\|_U^2\,dr.$$
Letting $\delta \to 0^+$ and replacing $s$ by $T-\eps$, the result is a consequence of Lemma \ref{stimaU}.
\end{proof}
In the next lemma we introduce a linear operator, whose importance is basic in the sequel. To this purpose it is useful to define:
\beq\label{proj}
\Pi:H\to H, \quad \Pi=\textrm{orthogonal projection onto } \overline{R(P_T^{1/2})}.
\eeq
\begin{lemma}\label{C(s)} Let $s\in [0,T[\,$. There exists $C(s)\in {\cal L}(H)$, with range contained in $\overline{R(P_T^{1/2})}$, such that $\Pi Q(T-\eps)^{1 / 2}\Phi(T-\eps,s)x\rightharpoonup C(s)x$ in $H$ for every $x\in H$, i.e.
\beq\label{c(s)}
\lim_{\eps\to 0^+} \langle Q(T-\eps)^{1 / 2}\Phi(T-\eps)x,y\rangle_H = \langle C(s)x,y\rangle_H \qquad \forall x\in H, \quad \forall y\in \overline{R(P_T^{1/2})}.
\eeq
\end{lemma}
\begin{proof} We start from \eqref{Riccint1}. Taking into account Lemma \ref{Mu_Mphi}, we get
\beyy
\lefteqn{\exists \lim_{\eps \to 0^+} \langle Q(T-\eps)^{1 / 2} \Phi(T-\eps,s)x,Q(T-\eps)^{1 / 2}U(T-\eps,s)y\rangle_H }\\
& & \qquad \ \ = \langle Q(s)x,y\rangle_H - \int_s^T \langle M(r)\Phi(r,s)x,U(r,s)y\rangle_H \,dr.
\eeyy
Replacing $s$ by $T-\vartheta$, with $0<\eps<\vartheta<T-s$, and $x$ by $\Phi(T-\vartheta,s)x$, we obtain by \eqref{transit}
\beq\label{conto_per_c(s)}\begin{array}{l}
\displaystyle \exists \lim_{\eps \to 0^+} \langle Q(T-\eps)^{1 / 2}\Phi(T-\eps,s)x,Q(T-\eps)^{1 / 2}U(T-\eps,T-\vartheta)y\rangle_H \\[2mm]
\displaystyle = \langle Q(T-\vartheta)\Phi(T-\vartheta,s)x,y\rangle_H - \int_{T-\vartheta}^T \langle M(r)\Phi(r,s)x,U(r,T-\vartheta)y\rangle_H \,dr.\end{array}\eeq
Now note that, by \eqref{b(s)} and \eqref{Riccint3},  
\beq\label{bound} \|Q(T-\eps)^{1 / 2}\Phi(T-\eps,s)\|_{{\cal L}(H)}\le K \qquad \forall x\in H,\quad \forall s\in [0,T[\,, \quad \forall \eps \in \,]0,T-s];
\eeq
hence, for each $x\in H$ and $s\in [0,T[\,$ we can find a sequence $\sigma=\{\sigma_k\}_{k\in \N}$, decreasing monotonically to $0$, and an element $v(s,x,\sigma) \in H$ such that
\beq\label{lim_prel}
\langle v(s,x,\sigma),y\rangle_H = \lim_{k\to \infty} \langle Q(T-\sigma_k)^{1 / 2} \Phi(t-\sigma_k,s)x, y\rangle_H \qquad \forall y\in H.
\eeq
Going back to \eqref{conto_per_c(s)},  using the strong convergence of $Q(\cdot)^{1 / 2}U(\cdot,T-\vartheta)$ to $P_T^{1/2}U(T,T-\vartheta)$, we get for every $y\in H$  
\beyy
\lefteqn{\langle v(s,x,\sigma),P_T^{1/2}U(T,T-\vartheta)y\rangle_H }\\[2mm]  
& & =\lim_{k\to \infty} \langle Q(T-\sigma_k)^{1 / 2} \Phi(t-\sigma_k)x, Q(T-\sigma_k)^{1 / 2} U(T-\sigma_k,T-\vartheta)y \rangle_H\\[1mm]
& & = \lim_{\eps \to 0^+} \langle Q(T-\eps)^{1 / 2}\Phi(T-\eps,s)x,Q(T-\eps)^{1 / 2}U(T-\eps,T-\vartheta)y\rangle_H \\[-1mm]
& & = \langle Q(T-\vartheta)\Phi(T-\vartheta,s)x,y\rangle_H - \int_{T-\vartheta}^T \langle M(r)\Phi(r,s)x,U(r,T-\vartheta)y\rangle_H \,dr.
\eeyy
As a consequence, since the integral term, by Lemma \ref{Mu_Mphi}, goes to $0$ as $\vartheta\to 0^+$,
$$\exists \lim_{\vartheta\to 0^+} \langle Q(T-\vartheta)^{1 / 2} \Phi(T-\vartheta,s)x,Q(T-\vartheta)^{1 / 2}y\rangle_H = \langle v(s,x,\sigma),P_T^{1/2}y\rangle_H\qquad \forall y\in H,$$
and also, using \eqref{bound} and strong convergence of $Q(\cdot)^{1 / 2 }$ to $P_T^{1/2}$,
$$\exists \lim_{\vartheta\to 0^+} \langle Q(T-\vartheta)^{1 / 2} \Phi(T-\vartheta,s)x,P_T^{1/2}y\rangle_H = \langle v(s,x,\sigma),P_T^{1/2}y\rangle_H\qquad\forall y\in H.$$
By density, recalling again \eqref{bound}, we obtain 
$$\exists \lim_{\vartheta\to 0^+} \langle Q(T-\vartheta)^{1 / 2} \Phi(T-\vartheta,s)x,z\rangle_H = \langle v(s,x,\sigma),z\rangle_H\qquad \forall z\in \overline{R(P_T^{1/2})}.$$
As the limit in the left-hand side is independent of the sequence $\sigma$ and is linear with respect to $x$, we deduce that there is a bounded (by \eqref{bound}), linear operator $\Gamma(s):H\to H$ such that 
$$\langle \Gamma(s)x,z\rangle_H = \langle v(s,x,\sigma),z\rangle_H =\lim_{\vartheta\to 0^+} \langle Q(T-\vartheta)^{1 / 2} \Phi(T-\vartheta,s)x,z\rangle_H \quad \forall z\in \overline{R(P_T^{1/2})}.$$
Finally, setting $C(s)x=\Pi\Gamma(s)x$ for every $x\in H$, 
it is clear that $\langle C(s)x,z\rangle_H = \langle \Gamma(s)x,z\rangle_H$ for every $z\in \overline{R(P_T^{1/2})}$, so that
$$\langle C(s)x,z\rangle_H =\langle v(s,x,\sigma),z\rangle_H =\lim_{\vartheta\to 0^+} \langle Q(T-\vartheta)^{1 / 2} \Phi(T-\vartheta,s)x,z\rangle_H \quad \forall z\in \overline{R(P_T^{1/2})}.$$
This proves the result. 
\end{proof} 
We now prove two basic lemmas relative to the behaviour of $C(\cdot)$ near $T$.
\begin{lemma}\label{range_limit} Let $C(s)$ be defined by  \eqref{c(s)}. Then:
\begin{description}
\item[(i)] \ $\displaystyle \lim_{t\to T^-} \|C(t)x-P_T^{1/2}x\|_H =0$ for every $x\in H$; 
\item[(ii)] \ $\displaystyle \overline{R(P_T^{1/2}})=\overline{\bigcup_{0\le t<T} R(C(t))}$. 
\end{description}
\end{lemma}
\begin{proof} In oreder to prove (i), as in the proof of Lemma \ref{stimaPhi}, we arrive to \eqref{a+c-2b}. With $x=y$, $\delta$ in place of $\eps$ and $T-\eps$ in place of $s$ (with $\eps>\delta$), we obtain
$$\begin{array}{l}
\displaystyle 0=\|Q(T-\delta)^{1/2}\Phi(T-\delta,T-\eps)x - Q(T-\delta)^{1/2}U(T-\delta,T-\eps)x\|_H^2\\[2mm]
\displaystyle \qquad+ \int_{T-\eps}^{T-\delta} \|M(r)^{1/2} [\Phi(r,T-\eps)x-U(r,T-\eps)x]\|_H^2\,dr\\[3mm]
\displaystyle \qquad+\int_{T-\eps}^{T-\delta} \left[\|N(r)^{-1/2}G(r)^*A(r)^+Q(r)\Phi(r,T-\eps)x\|^2_U\right.\\[2mm]
\displaystyle \qquad\qquad \left.-\|N(r)^{-1/2}G(r)^*A(r)^+Q(r)U(r,T-\eps)x\|^2_U\right]\,dr,
\end{array}$$  
Letting $\delta \to 0^+$, we see that
$$\begin{array}{l}
\displaystyle \exists \lim_{\delta \to 0^+}\|Q(T-\delta)^{1/2}\Phi(T-\delta,T-\eps)x - Q(T-\delta)^{1/2}U(T-\delta,T-\eps)x\|_H^2\\[2mm]
\displaystyle \qquad = - \int_{T-\eps}^T \|M(r)^{1/2} [\Phi(r,T-\eps)x-U(r,T-\eps)x]\|_H^2\,dr\\[3mm]
\displaystyle \qquad \quad-\int_{T-\eps}^T \left[\|N(r)^{-1/2}G(r)^*A(r)^+Q(r)\Phi(r,T-\eps)x\|^2_U\right.\\[3mm]
\displaystyle \qquad\quad \quad\left.-\|N(r)^{-1/2}G(r)^*A(r)^+Q(r)U(r,T-\eps)x\|^2_U\right]\,dr.
\end{array}$$  
By \eqref{c(s)} we get
$$\begin{array}{l}
\displaystyle \|C(T-\eps)x-P_T^{1/2}U(T,T-\eps)x\|_H^2 \le -\int_{T-\eps}^T \|M(r)^{1/2} [\Phi(r,T-\eps)x-U(r,T-\eps)x]\|_H^2\,dr \\[2mm]
\displaystyle \qquad-\int_{T-\eps}^T \left[\|N(r)^{-1/2}G(r)^*A(r)^+Q(r)\Phi(r,T-\eps)x\|^2_U\right.\\[2mm]
\displaystyle \qquad\qquad \left.-\|N(r)^{-1/2}G(r)^*A(r)^+Q(r)U(r,T-\eps)x\|^2_U\right]\,dr.
\end{array}$$  
Finally, we let $\eps \to 0^+$: by Lemmas \ref{Mu_Mphi}, \ref{stimaPhi} and \ref{stimaU} we educe that
$$\lim_{\eps\to 0^+} \|C(T-\eps)x-P_T^{1/2}U(T,T-\eps)x\|_H^2=0;$$
as $P_T^{1/2}U(T,T-\eps)x \to P_T^{1/2}x$ in $H$, we conclude that $C(T-\eps)x \to P_T^{1/2}x$, thus proving (i).\\[1mm]
Concerning (ii), the inclusion $\supseteq$ is easy, since Lemma \ref{C(s)} implies that
$$\bigcup_{0\le t<T} R(C(t)) \subseteq \overline{R(P_T^{1/2})}.$$
To prove the reverse inclusion, fix $z\in \overline{R(P_T^{1/2})}$ and $\eps>0$: then there exists $x\in H$ such that $\|P_T^{1/2}x-z\|_H <\eps$. By (i), $C(t)x \to P_T^{1/2}x$ as $t\to T^-$, so that there is $t_0\in [0,T[\,$ such that 
$$\|C(t)x-P_T^{1/2}x\|_H < \eps \qquad \forall t\in [t_0,T[\,.$$
This shows that
$$\overline{R(P_T^{1/2})} \subseteq \overline{ \bigcup_{0\le t_0<T}\ \bigcap_{t_o\le t<T} R(C(t))}.$$
However, if $t_0\le t<t'<T$ it holds, by \eqref{transit}, 
$$C(t)x =\textrm{w-}\hspace{-1mm}\lim_{\eps\to 0^+}Q(T-\eps^{1/2}\Phi(T-\eps,t)x = \textrm{w-}\hspace{-1mm}\lim_{\eps\to 0^+}Q(T-\eps^{1/2}\Phi(T-\eps,t')\Phi(t',t)x = C(t')\phi(t',t)x,$$ 
so that $R(C(t))\subseteq R(C(t'))$. As a consequence,
$$\overline{R(P_T^{1/2})} \subseteq \overline{ \bigcup_{0\le t<T} R(C(t))}.$$
\end{proof}
\begin{lemma}\label{legameBC} Let $B(s)$ and $C(s)$ be defined by \eqref{b(s)xy} and \eqref{c(s)}. Then for every $x,y\in H$, $s\in [0,T[\,$ and $t\in [s,T[\,$ we have 
\beq\label{lim_key}
\lim_{\eps\to 0^+} \left|\langle C(t)x,P_T^{1/2} U(T,T-\eps)\Phi(T-\eps,s)y\rangle_H-\langle B(t)x,\Phi(t,s)y\rangle_H\right|=0.
\eeq
\end{lemma}
\begin{proof} 
We fix $s\in [0,T[\,$, $t\in[s,T[\,$ and $x,y\in H$. Let us analyze the quantity to be estimated. To begin with, we have for $\eps\in \,]0,T-t[\,$
\beq\label{stima1}
\begin{array}{l}
\left|\langle C(t)x,P_T^{1/2} U(T,T-\eps)\Phi(T-\eps,s)y\rangle_H-\langle B(t)x,\Phi(t,s)y\rangle_H\right|\\[2mm]
\displaystyle\le \Big|\langle C(t)x,P_T^{1/2} U(T,T\hspace{-0.5mm}-\hspace{-0.5mm}\eps)\Phi(T\hspace{-0.5mm}-\hspace{-0.5mm}\eps,s)y\rangle_H\hspace{-0.5mm}-\hspace{-0.5mm} \langle Q(T\hspace{-0.5mm}-\hspace{-0.5mm}\eps)\Phi(T\hspace{-0.5mm}-\hspace{-0.5mm}\eps,t)x,\Phi(T\hspace{-0.5mm}-\hspace{-0.5mm}\eps,s)y\rangle_H \Big| \\[2mm]
\displaystyle \quad \quad \quad + \Big|\langle Q(T-\eps)\Phi(T-\eps,t)x,\Phi(T-\eps,s)y\rangle_H-\langle B(t)x,\Phi(t,s)y\rangle_H\Big|.
\end{array}
\eeq
In order to estimate the first term on the right-hand side of \eqref{stima1}, we start from equation \eqref{Riccint1}: we first replace there $\eps$ by $\delta$ and then choose $T-\eps$ (with $0<\delta<\eps$) in place of $s$, $\Phi(T-\eps,t)x$ in place of $x$ and $\Phi(T-\eps,s)y$ in place of $y$. Then we find, using \eqref{transit}, 
\beq\label{stima2}
\begin{array}{l} \langle Q(T-\eps)\Phi(T-\eps,t)x,\Phi(T-\eps,s)y\rangle_H\\[2mm]
\displaystyle \qquad =\langle Q(T-\delta)^{1 / 2}\Phi(T-\delta,t)x,Q(T-\delta)^{1 / 2}U(T-\delta,T-\eps)\Phi(T-\eps,s)y\rangle_H\\[1mm]
\displaystyle \qquad \quad+\int_{T-\eps}^{T-\delta} \langle M(r)\Phi(r,t)x,U(r,T-\eps)\Phi(T-\eps,s)y\rangle_H\,dr.
\end{array}
\eeq
As $\delta\to 0^+$, by strong convergence of $Q(T-\delta)^{1/2} U(T-\delta,T-\eps)y$ to $P_T^{1/2}U(T,T-\eps)y$ and by \eqref{bound}, \eqref{c(s)}, we deduce  
\beq \label{base_cu}\begin{array}{l} 
\hspace{-5mm}\displaystyle \left|\langle Q(T-\eps)\Phi(T-\eps,t)x,\Phi(T-\eps,s)y\rangle_H -\langle C(t)x,P_T^{1/2}U(T,T-\eps)\Phi(T-\eps,s)y\rangle_H\right| \\[2mm]
\displaystyle \qquad=\left|\int_{T-\eps}^{T} \langle M(r)\Phi(r,t)x,U(r,T-\eps)\Phi(T-\eps,s)y\rangle_H\,dr\right|\\[3mm]
\displaystyle \qquad= \left|\int_{T-\eps}^{T} \langle M(r)\Phi(r,t)x,y_\eps(r;s,y)\rangle_H\,dr \right| \\[4mm]
\displaystyle\qquad \le K_0 \|\bar{y}(\cdot;t,x)\|_{L^2(T-\eps,T;H)}\|y_\eps(\cdot;s,y)\|_{L^2(s,T;H)}\end{array} 
\eeq
where $K_0$ is a constant; since $y_\eps(\cdot;s,y)$ is bounded in $L^2(s,T;H)$ by a constant $K_y$\,, we proceed as in the proof of Lemma \ref{Mu_Mphi}, and we deduce
\beq\label{base_cu2}
\begin{array}{l}
\hspace{-5mm}\displaystyle \left|\int_{T-\eps}^{T} \langle M(r)\Phi(r,t)x,U(r,T-\eps)\Phi(T-\eps,s)y\rangle_H\,dr\right| \le \\[4mm]
\displaystyle \hspace{-5mm}\qquad \qquad \le  K_y\,\eps^\alpha\|\bar{y}(\cdot;t,x)\|_{L^{\frac2{1-2\alpha}}(T-\eps,T;H)}\le  CK_y\eps^{\alpha}\|x\|_H.\end{array}
\eeq 
The second term in the right-hand side of \eqref{stima1} can be estimated by subtracting \eqref{Riccint3bis} from \eqref{Riccint2}, having replaced in both equations $s$ by $t$ and $y$ by $\Phi(t,s)y$: indeed, we have, using \eqref{transit}, \eqref{controlbar} and \eqref{transitu}, 
\beyy
0 & = & \langle Q(T-\eps)\Phi(T-\eps,t)x,\Phi(T-\eps,s)y\rangle_H-\langle B(t)x,\Phi(t,s)y\rangle_H - \\
& - & \int_{T-\eps}^T \langle M(r)\Phi(r,t)x,\Phi(r,s)y\rangle_H \,dr - \int_{T-\eps}^T \langle N(r)\bar{u}(r;t,x),\bar{u}(r;s,y)\rangle_U\,dr.
\eeyy
Hence, proceeding as before, we get
\beq\label{stima3} 
\begin{array}{l}\big|\langle Q(T-\eps)\Phi(T-\eps,t)x,\Phi(T-\eps,s)y\rangle_H-\langle B(t)x,\Phi(t,s)y\rangle_H\big|\le\\[2mm]
\displaystyle \qquad\le K\eps^{2\alpha} \|\bar{y}(\cdot;t,x)\|_{L^{\frac1{1-2\alpha}}(T-\eps,T;H)}\|\bar{y}(\cdot;s,y)\|_{L^{\frac1{1-2\alpha}}(T-\eps,T;H)}+ \\[3mm]
\displaystyle \qquad +K\|\bar{u}(\cdot;t,x)\|_{L^2(T-\eps,T;U)}\|\bar{u}(\cdot;s,y)\|_{L^2(T-\eps,T;U)} \le K\|x\|_H\left(\eps^{2\alpha}\|y\|_H+\omega_y(\eps)\right),
\end{array}
\eeq
where $K$ is an absolute constant and $\omega_y(\eps)$ decreases monotonically to $0$ as $\eps\to 0^+$. By \eqref{stima1}, \eqref{base_cu} and \eqref{stima3} we obtain the desired conclusion.
\end{proof}
The following statement is an important variant of the preceding lemma.
\begin{lemma}\label{legameBCeps} Let $C(s)$ be defined by \eqref{c(s)}. Then for every $s\in [0,T]\,$ we have 
$$\lim_{\eps \to 0^+} \langle C(T-\eps)x,P_T^{1/2} U(T,T-\eps)\Phi(T-\eps,s)y\rangle_H =\langle P_T^{1/2}x,C(s)y \rangle_H \qquad \forall x,y \in H.$$
\end{lemma}
\begin{proof}
We go back once again to \eqref{Riccint1}. We replace $\eps$ by $\delta$ (with $0<\delta <\eps$), $s$ by $T-\eps$ and $y$ by $\Phi(T-\eps,s)y$: we get 
\beyy
\lefteqn{\langle Q(T-\eps)x,\Phi(T.\eps,s)y\rangle_H }\\
& & = \langle Q(T-\delta)^\frac12 \Phi(T-\delta,T-\eps)x, Q(T-\delta)^\frac12 U(T-\delta,T-\eps)\Phi(T-\eps,s)y\rangle_H \\
& & \qquad \qquad\qquad\qquad+\int_{T-\eps}^{T-\delta}\langle M(r)\Phi(r,T-\eps)x,U(r,T-\eps)\Phi(T-\eps,s)y\rangle_H\, dr,
\eeyy
and by \eqref{bound}, strong convergence of $Q(T-\delta)^\frac12 U(T-\delta,T-\eps)$  and \eqref{c(s)}, as $\delta \to 0^+$,
\beyy
\lefteqn{\langle Q(T-\eps)^\frac12 x,Q(T-\eps)^\frac12 \Phi(T-\eps,s)y\rangle_H =\langle C(T-\eps)x,P_T^{1/2}U(T,T-\eps)\Phi(T-\eps,s)y\rangle_H }\\
& & \qquad \qquad \qquad \qquad \qquad \qquad + \int_{T-\eps}^T \langle M(r)\Phi(r,T-\eps)x,U(r,T-\eps)\Phi(T-\eps,s)y\rangle_H \,dr.
\eeyy
The last integral tends to $0$ as $\eps\to 0^+$ by Lemma \ref{Mu_Mphi}, arguing as in \eqref{base_cu2}. Thus, using again strong convergence of $Q(\cdot)^\frac12$, \eqref{bound} and \eqref{bound}, we conclude that 
\beyy
\lefteqn{\lim_{\eps \to 0^+} \langle C(T-\eps)x,P_T^{1/2}U(T,T-\eps)\Phi(T-\eps,s)y\rangle_H }\\
& & = \lim_{\eps \to 0^+} \langle Q(T-\eps)^\frac12 x,Q(T-\eps)^\frac12 \Phi(T-\eps,s)y\rangle_H =\langle P_T^{1/2}x,C(s)y \rangle_H\,,
\eeyy
which is our claim.
\end{proof}
The following lemma is crucial, in order to obtain the key relation $\bar{u}(\cdot;s,x)\in D(L_{sT})$.
\begin{lemma}\label{c(s)vabene} Let $\bar{u}(\cdot;s,x)$ and $C(s)$ be defined by \eqref{controlbar} and \eqref{c(s)}. For every $x\in H$ and $s\in [0,T[\,$ we have $P_T^{1/2}C(s)x\in D(L_{sT}^*)$ and
\beq\label{u_bar_new}
\bar{u}(\cdot,s;x) = -N(\cdot)^{-1}L_{sT}^*P_T^{1/2}C(s)x-N(\cdot)^{-1}L_s^*[M(\cdot)\Phi(\cdot,s)x].
\eeq
\end{lemma}
\begin{proof}
We start from the integral equation \eqref{Riccint1}, which can be rewritten in strong form, with $s$ replaced by $t$ and $x$ replaced by $\Phi(t,s)x$, as
$$Q(t)\Phi(t,s)x= U(T-\eps,t)^*Q(T-\eps)\Phi(T-\eps,s)x+\int_t^{T-\eps} U(r,t)^*M(r)\Phi(r,s)x\,dr .$$
We operate with $G(t)^*A(t)^*$ in both members: by \eqref{controlbar} we obtain 
\beq\label{eq_per_u}
\begin{array}{lcl}
N(t)\bar{u}(t,s,x) & = & G(t)^*A(t)^*Q(t)\Phi(t,s)x \\[2mm]
& = & G(t)^*A(t)^*U(T-\eps,t)^*Q(T-\eps)\Phi(T-\eps,s)x\\
& & \quad \displaystyle +\int_t^{T-\eps} G(t)^*A(t)^*U(r,t)^*M(r)\Phi(r,s)x\,dr .
\end{array} 
\eeq
The second term in the last member of \eqref{eq_per_u} is exactly the function $-L_s^*[\chi_{[s,T-\eps]}M(\cdot)\Phi(\cdot,s)x]$ evaluated in $t$, which converges as $\eps\to 0^*$ to $-L_s^*[M(\cdot)\Phi(\cdot,s)x]$ in $L^2(s,T;U)$: indeed, by \cite[Lemma 4.4(iv)]{AT1},
$$
\begin{array}{l}
\displaystyle \int_s^T \left\|L_s^*[\chi_{[s,T-\eps]}M(\cdot)\Phi(\cdot,s)x](t)-L_s^*[M(\cdot)\Phi(\cdot,s)x](t)\right\|^2_U\,dt\\
\qquad \displaystyle = \int_s^T \left\| L_s^*[\chi_{[T-\eps,T]}M(\cdot)\Phi(\cdot,s)x]](t)\right\|^2_U \,dt \le \left[\int_{T-\eps}^T \|M(r)\Phi(r,s)x\|_u^q\,dr\right]^{\frac2{q}},
\end{array}
$$
with $q=\frac2{1-2\alpha}\,$, and the last quantity goes to $0$ as $\eps \to 0^*$.\\
By difference, the first term in the last member of \eqref{eq_per_u} belongs to $L^2(s,T;U)$ and it converges to $ G(\cdot)^*A(\cdot)^*U(T,\cdot)^*P_T^{1/2}C(s)x$ in $L^2(s,t;U)$. By definition (see \eqref{LsT*}), this means $P_T^{1/2}C(s )x \in D(L_{sT}^*)$ and formula \eqref{u_bar_new} follows.
\end{proof}
We now state the key result of this proof.
\begin{proposition}\label{u_in_D(LsT)} Let $\bar{u}(\cdot;s,x)$ and $C(s)$ be defined by \eqref{controlbar} and \eqref{c(s)}. For $0\le s<T$ and $x\in H$ we have
$\bar{u}(\cdot;s,x)\in D(L_{sT})$ and
$$C(s)x = P_T^{1/2}U(T,s)x + P_T^{1/2}L_{sT}\bar{u}(\cdot;s,x).$$
\end{proposition}
\begin{proof} First of all, we note that, by Lemma \ref{c(s)vabene} and \eqref{LsT*}, it holds $P_T^{1/2}C(t)x\in D(L_{tT}^*) \subseteq D(L_{sT}^*)$ for every $s\in[0,t]$. Thus, recalling \eqref{u_eps_y_eps}, we start from the limit in \eqref{lim_key} and rewrite it as 
$$\begin{array}{l}
\langle B(t)x,\Phi(t,s)y\rangle_H = \displaystyle \lim_{\eps \to 0^-} \langle C(t)x,P_T^{1/2}U(T,s)y+P_T^{1/2}L_{sT}u_\eps(\cdot;s,x)\rangle_H\\
\quad \quad =\displaystyle \langle C(t)x, P_T^{1/2}U(T,s)y\rangle_H + \lim_{\eps \to 0^-}\langle L_{sT}^*P_T^{1/2}C(t)x,u_\eps(\cdot;s,y)\rangle_{L^2(s,T;U)}\\
\quad \quad =\langle C(t)x, P_T^{1/2}U(T,s)y\rangle_H + \langle L_{sT}^*P_T^{1/2}C(t)x,\bar{u}(\cdot;s,y)\rangle_{L^2(s,T;U)}\\[1mm]
\quad \quad =\langle x, C(t)^*P_T^{1/2}U(T,s)y + [L_{sT}^*P_T^{1/2}C(t)]^*\bar{u}(\cdot;s,y)\rangle_H\,,\end{array}
$$
where, by Lemma \ref{c(s)vabene}, the operator $L_{sT}^*P_T^{1/2}C(t)$ is closed with domain $H$; hence it is  bounded, and of course its adjoint is bounded, too. Since $x\in H$ is arbitrary, we deduce
$$B(t)\Phi(t,s)y = C(t)^*P_T^{1/2}U(T,s)y +[L_{sT}^*P_T^{1/2}C(t)]^*\bar{u}(\cdot;s,y).$$
We note that the operator $L_{sT}^*P_T^{1/2}$ is obviously closed and its domain $$D(L_{sT}^*P_T^{1/2})\supseteq \left[\bigcup_{s\le t<T} R(C(t))\right] \cup \ker P_T^{1/2}$$
is dense in $H$ by Lemma \ref{range_limit}(ii); moreover $C(t)$ is bounded. So we may apply \cite[Theorem 13.2]{R},obtaining $[L_{sT}^*P_T^{1/2}C(t)]^*=C(t)^*[L_{sT}^*P_T^{1/2}]^*$. Thus we may write
\beq\label{b(t)fi(t,s)}
B(t)\Phi(t,s)y = C(t)^*\left[P_T^{1/2}U(T,s)y +[L_{sT}^* P_T^{1/2}]^*\bar{u}(\cdot;s,y)\right];
\eeq 
In particular, we observe that 
$$L^2(s;T;U) = D([L_{sT}^*P_T^{1/2}C(t)]^*)=D(C(t)^*[L_{sT}^* P_T^{1/2}]^*)\subseteq D([L_{sT}^* P_T^{1/2}]^*).$$
It is easy to verify that the operator $[L_{sT}^* P_T^{1/2}]^*$ is closed; since it has the whole space $L^2(s,T;U)$ as its domain, it must be bounded, i.e. $[L_{sT}^* P_T^{1/2}]^*\in {\cal L}(L^2(s,T;U),H)$. Moreover, since $P_T^{1/2}\in {\cal L}(H)$ and $L_{sT}$ is densely defined, again by \cite[Theorem 13.2]{R} we have $[L_{sT}^* P_T^{1/2}]^*=[(P_T^{1/2} L_{sT})^*]^*$; so, $[L_{sT}^* P_T^{1/2}]^*$ is a bounded extension of the closed operator $P_T^{1/2}L_{sT}$. 
Using this information, we may write, as $\eps \to 0^+$,
$$\begin{array}{lcl}
P_T^{1/2}L_{sT}u_\eps=P_T^{1/2}U(T,T-\eps)\Phi(T-\eps,s)x  & = & P_T^{1/2}U(t,s)x + P_T^{1/2}L_{sT}u_\eps(\cdot;s,x) \\[1mm]
& = & P_T^{1/2}U(t,s)x + [L_{sT}^*P_T^{1/2}]^* u_\eps(\cdot;s,x) \\[1mm]
& \to & P_T^{1/2}U(t,s)x + [L_{sT}^*P_T^{1/2}]^*\bar{u}(\cdot;s,x).
\end{array}$$
Thus,
$$u_\eps(\cdot;s,x)\to \bar{u}(\cdot;s,x) \ \textrm{in} \ L^2(s,T;U), \qquad 
P_T^{1/2}L_{sT}u_\eps(\cdot;s,x) \to [L_{sT}^*P_T^{1/2}]^*\bar{u}(\cdot;s,x) \ \textrm{in} \ H,$$
so that $\bar{u}(\cdot;s;x)\in D(L_{sT})$ and $P_T^{1/2}L_{sT}\bar{u}(\cdot;s,x)= [L_{sT}^*P_T^{1/2}]^*\bar{u}(\cdot;s,x)$.\\
Set momenrtarily
$$E(s)x= \lim_{\eps\to 0^+} P_T^{1/2}U(T,T-\eps)\Phi(T-\eps,s)x.$$
By \eqref{lim_key} with $t=s$, using Lemma \ref{legameBCeps}, we deduce for every $x\in H$ 
$$\langle P_T^{1/2}x,,C(s)y\rangle_H=\lim_{\eps\to 0^+} \langle C(T-\eps)x,P_T^{1/2}U(T,T-\eps)\Phi(T-\eps,s)y\rangle_H = \langle P_T^{1/2}x,E(s)y\rangle_H\,.$$ 
Since $E(S)$ and $C(s)$ are bounded operators, by density we have
$$\langle z,C(s)y\rangle_H = \langle z,E(s)y\rangle_H \qquad \forall z\in \overline{R(P_T^{1/2})}.$$
As both $E(s)$ and $C(s)$ have range contained  in $\overline{R(P_T^{1/2})}$, we deduce $E(s)=C(s)$, i.e.
$$P_T^{1/2}U(T,T-\eps)\Phi(T-\eps,s)x \to C(s)x \quad \textrm{in} \ H \quad \textrm{as} \ \eps \to 0^+,$$
and the result follows.
\end{proof}
\begin{corollary}\label{B=C*C} Let $C(s)$ and $B(s)$ be defined by \eqref{c(s)} and \eqref{b(s)xy}. For every $s\in[0,T[\,$ we have the equality
$$B(s)=C(s)^*C(s);$$ 
consequently
$$Q(T-\eps)^{1/2}\Phi(T-\eps,s)x \to C(s)x \quad \textrm{in}\ H \quad \textrm{as} \ \eps \to 0^+.$$
\end{corollary}
\begin{proof}
Indeed, by Lemma \ref{legameBC} with $t=s$, we have for every $x,y\in H$,
$$\langle B(s)x,y\rangle_H = \lim_{\eps\to 0^+} \langle C(s)x,P_T^{1/2}U(T,T-\eps)\Phi(T-\eps,s)y\rangle_H \,,$$
and by Proposition \ref{u_in_D(LsT)},
$$\langle B(s)x,y\rangle_H = \langle C(s)x,C(s)y\rangle_H\,,$$
and this means $B(s)=C(s)^*C(s)$. Hence, as $\eps \to 0^+$, 
$$Q(T-\eps)^{1/2}\Phi(T-\eps,s)x \rightharpoonup C(s)x, \quad \|Q(T-\eps)^{1/2}\Phi(T-\eps,s)x\|_H^2 \to \langle B(s)x,x\rangle_H = \|C(s)x\|_H^2\,.$$
This implies $Q(T-\eps)^{1/2}\Phi(T-\eps,s)x \to C(s)x$ in $H$ as $\eps\to 0^+$.
\end{proof}
We can now conclude the proof of Theorem \ref{unic} {\bf (b)}. In fact, since $\bar{u}(\cdot;s,x)\in D(L_{sT})$, we have, recalling \eqref{Riccint3} (and writing $\bar{u}(\cdot;s,x)=\bar{u}$ and $\bar{y}(\cdot;s,x)=\bar{y}$)
$$\begin{array}{lcl}
\langle P(s)x,x\rangle_H & = & J_s(\widehat{u}(\cdot;s,x)) \le J_s(\bar{u}(\cdot;s,x)) \\[1mm]
& = & \displaystyle \int_s^T \langle M(r)\bar{y}(r),\bar{y}(r)\rangle_H\,dr + \displaystyle \int_s^T \langle N(r)\bar{u}(r),\bar{u}(r)\rangle_U\,dr + \langle B(s)x,x\rangle_H \\[1mm]
& =& \langle Q(s)x,x\rangle_H\,.
\end{array}$$ 
This proves {\bf (b)} in \eqref{tesi}. \\[2mm]
{\bf Proof of Theorem \ref{reg_Py}.} The argument is very easy: $P(\cdot)$ is the unique solution of the Riccati equation; as the optimal control $\widehat{u}$ is given by \eqref{optcontr_feedback}, it coincides with the function $\bar{u}$ given by \eqref{controlbar}, and hence the optimal state $\widehat{y}=\widehat{\Phi}(\cdot,s)x$ coincides with $\bar{y}=\Phi(\cdot,s)x$. On the other hand, by Corollary \ref{B=C*C}, $t\mapsto Q(t)^{1/2}\Phi(t,s)x=P(t)^{1/2}\widehat{\Phi}(t,s)x$ is continuous in $[s,T]$ with $P_T^{1/2}\Phi(T,s)x= C(s)x = P_T^{1/2}[U(T,s)x+ L_{sT}\bar{u}(\cdot;s,x)]$. This proves Theorem \ref{reg_Py}.
\section*{Appendix}
\appendix
\section{Proof of Lemma \ref{idcontr}}\label{A}
Let us split the terms on the right-hand side of equation \eqref{idbase}. Recalling \eqref{state_Ls},  we have $y(t) = U(t,s)x+[L_su](t)$, with $L_s$ given by \eqref{Ls}, and hence
\beyy 
\lefteqn{-\int_s^{T-\eps} \|M(r)^{1/2}y(r)\|_H^2\,dr = -\int_s^{T-\eps} \|M(r)^{1/2}U(r,s)x\|_H^2\,dr } \\
& & - 2\textrm{Re}\int_s^{T-\eps} \langle M(r)U(r,s)x,[L_su](r)\rangle_H\,dr - \int_s^{T-\eps} \|M(r)^{1/2}[L_su](r)\|_H^2\,dr \\[2mm]
& & =:M_1 + M_2 + M_3.
\eeyy
Similarly
\beyy 
\lefteqn{\int_s^{T-\eps} \|N(r)^{1/ 2}u(r)-N(r)^{- 1 / 2}G(r)^*A(r)^*Q(r)y(r)\|_U^2 \,dr } \\
& & = \int_s^{T-\eps} \|N(r)^{1/2}u(r)\|_U^2\,dr - 2 \textrm{Re} \int_s^{T-\eps} \langle u(r),G(r)^*A(r)^*Q(r)y(r)\rangle_U\,dr \\
& & + \int_s^{T-\eps} \|N(r)^{-1/2}G(r)^*A(r)^*Q(r)y(r)\|_U^2 \,dr,
\eeyy
and using again \eqref{state_Ls} we write
\beyy 
\lefteqn{\int_s^{T-\eps} \|N(r)^{1/2}u(r)-N(r)^{-1/2}G(r)^*A(r)^*Q(r)y(r)\|_U^2 \,dr } \\
& & = \int_s^{T-\eps} \|N(r)^{1/2}u(r)\|_U^2\,dr - 2 \textrm{Re} \int_s^{T-\eps} \langle u(r),G(r)^*A(r)^*Q(r)U(r,s)x\rangle_U\,dr \\
& & - 2 \textrm{Re} \int_s^{T-\eps} \langle u(r),G(r)^*A(r)^*Q(r)[L_su](r)\rangle_U\,dr \\
& & + \int_s^{T-\eps} \|N(r)^{-1/2}G(r)^*A(r)^*Q(r)U(r,s)x\|_U^2\,dr \\
& & + 2 \textrm{Re} \int_s^{T-\eps} \langle N(r)^{-1}G(r)^*A(r)^*Q(r)U(r,s)x,G(r)^*A(r)^*Q(r)[L_su](r)\rangle_U\,dr  \\
& & + \int_s^{T-\eps} \|N(r)^{-1/2}G(r)^*A(r)^*Q(r)[L_su](r)\|_U^2\,dr \\[2mm]
& & =: N_1 + N_2 + N_3 + N_4 + N_5 + N_6\,.
\eeyy
Hence in the right member of \eqref{idbase} the term $N_1$ cancels and we get
\beq \label{decompos}
\begin{array}{l} \displaystyle -\int_s^{T-\eps} \|M(r)^{1 / 2}y(r)\|_H^2\,dr -\int_s^{T-\eps} \|N(r)^{1 / 2}u(r)\|_U^2\,dr + \\[4mm]
\qquad \displaystyle +\int_s^{T-\eps} \|N(r)^{1 / 2}u(r)-N(r)^{- 1 / 2}G(r)^*A(r)^*Q(r)y(r)\|_U^2 \,dr = \\[4mm]
\qquad = M_1 + M_2 + M_3 + N_2 + N_3 + N_4 + N_5+N_6\,.\end{array}
\eeq
Now we note that, by (\ref{Riccintw}) with $x=y$,
\beyy
\lefteqn{M_1 + N_4 =-\int_s^{T-\eps}\|M(r)^{1 / 2}U(r,s)x\|_H^2\,dr }\\
& & + \int_s^{T-\eps} \|N(r)^{- 1 / 2}G(r)^*A(r)^*Q(r)U(r,s)x\|_U^2\,dr \\[2mm]
& &= -\langle Q(s)x,x\rangle_H + \langle Q(T-\eps)U(T-\eps,s)x,U(T-\eps,s)x\rangle_H\,.
\eeyy
Next, we set
\beq\label{Ls_eps}
L^{*\eps}_s v](t) = -G(t)^*A(t)^*\int_t^{T-\eps} U(\sigma,t)^*v(\sigma)\,d\sigma, \quad v\in L^2(s,T-\eps;H);
\eeq
this is the adjoint of $L_s:L^2(s,T-\eps;U)\to L^2(s,T-\eps;H)$. For this operator we need the following property:
\begin{lemma}\label{calcoloLs} Let $L_s$, $L_s^\eps$ be given by \eqref{Ls}, \eqref{Ls_eps}. If $\Psi\in L^\infty(s,T-\eps;\Sigma(H))$, and $v\in L^2(s,T-\eps;U)$, then we have
\beyy
\lefteqn{\int_s^{T-\eps} \langle [L^{*\eps}_s\Psi(\cdot)L_sv](r),v(r)\rangle_U \,dr}\\
& & =- 2\textrm{{\em Re}}\int_s^{T-\eps} \left\langle G(q)^*A(q)^*\left[ \int_q^{T-\eps} U(r,q)^*\Psi(r)U(r,q)dr\right][L_sv](q),v(q)\right\rangle_U\,dq.
\eeyy
\end{lemma}
\begin{proof} Indeed we have
\beyy
\lefteqn{\int_s^{T-\eps} \langle [L^{*\eps}_s\Psi(\cdot)L_sv](r),v(r)\rangle_U \,dr = \int_s^{T-\eps} \langle \Psi(r)[L_sv](r),[L_sv](r)\rangle_H} \\
& & = \int_s^{T-\eps} \int_s^r \int_s^r \langle \Psi(r)U(r,\sigma)A(\sigma)G(\sigma)v(\sigma), U(r,q)A(q)G(q)v(q)\rangle_H\,d\sigma dq dr.
\eeyy
The integrand in the last term is a symmetric function $a(\sigma,q)$: we rewrite this term as
\beyy
\lefteqn{\int_s^{T-\eps} \int_s^r \int_s^r a(\sigma, q)\,d\sigma dq dr }\\ 
& & = \int_s^{T-\eps} \int_s^r \left[\int_s^q a(\sigma, q)\,d\sigma + \int_q^r a(\sigma, q)\,d\sigma\right] dq dr \\
& & = \int_s^{T-\eps} \left[\int_s^r \int_s^q a(\sigma, q)\,d\sigma dq + \int_s^r \int_s^\sigma a(\sigma, q)\,dqd\sigma \right]dr  \\
& & = 2 \textrm{Re} \int_s^{T-\eps} \int_s^r \int_s^q a(\sigma, q)\,d\sigma dqdr,
\eeyy
so that 
\beyy
\lefteqn{\int_s^{T-\eps} \langle [L^{*\eps}_s\Psi(\cdot)L_sv](r),v(r)\rangle_U \,dr }\\
& & = 2 \textrm{Re} \int_s^{T-\eps} \int_s^r \int_s^q \langle \Psi(r)U(r,\sigma)A(\sigma)G(\sigma)v(\sigma), U(r,q)A(q)G(q)v(q)\rangle_H\,d\sigma dqdr\\
& & = 2 \textrm{Re} \int_s^{T-\eps} \int_q^{T-\eps} \Bigg\langle G(q)^*A(q)^*U(r,q)^*\Psi(r) U(r,q) \\
& & \qquad \times \int_s^q U(q,\sigma)A(\sigma)G(\sigma)v(\sigma)\,d\sigma, v(q)\Bigg\rangle_U\,drdq\\
& & = -2\textrm{Re} \int_s^{T-\eps} \Bigg\langle G(q)^*A(q)^*\left[\int_q^{T-\eps} U(r,q)^*\Psi(r) U(r,q)\,dr \right] [L_sv](q),v(q)\Bigg\rangle_U\,dq.
\eeyy 
\end{proof}
Let us go back to \eqref{decompos}: we have
\beyy 
\lefteqn{M_2 + N_5 +N_2 = - 2\textrm{Re}\int_s^{T-\eps} \langle M(r)U(r,s)x,[L_su](r)\rangle_H\,dr }\\
& & \quad + 2 \textrm{Re} \int_s^{T-\eps} \langle N(r)^{-1}G(r)^*A(r)^*Q(r)U(r,s)x,G(r)^*A(r)^*Q(r)[L_su](r)\rangle_U\,dr  \\
& & \quad - 2 \textrm{Re} \int_s^{T-\eps} \langle u(r),G(r)^*A(r)^*Q(r)U(r,s)x\rangle_U\,dr  \\
& & = - 2 \textrm{Re}\int_s^{T-\eps} \left\langle [L^{*\eps}_sM(\cdot)U(\cdot,s)x](r),u(r)\right\rangle_U\,dr \\
& & \quad + 2 \textrm{Re}\int_s^{T-\eps} \left\langle [L^{*\eps}_s Q(\cdot)A(\cdot)G(\cdot)N(\cdot)^{-1}G(\cdot)^*A(\cdot)^*Q(\cdot)U(\cdot,s)x](r), u(r)\right\rangle_U\,dr \\
& & \quad - 2 \textrm{Re} \int_s^{T-\eps} \langle G(r)^*A(r)^*Q(r)U(r,s)x, u(r)\rangle_U\,dr .
\eeyy
Using \eqref{Ls_eps} and \eqref{Riccintw} we find 
\beyy 
\lefteqn{M_2 + N_5 + N_2 =2 \textrm{Re} \int_s^{T-\eps} \Bigg\langle G(r)^*A(r)^* \Bigg[ \int_r^{T-\eps} \big[U(\sigma,r)^*M(\sigma)U(\sigma,r)}\\
& & - U(\sigma,r)^*Q(\sigma)A(\sigma)G(\sigma)N(\sigma)^{-1}G(\sigma)^*A(\sigma)^*Q(\sigma)U(\sigma,r)\big]d\sigma  \\
& & -Q(r)\Bigg]  U(r,s)x,u(r)\Bigg\rangle_U\,dr \\
& & = - 2 \textrm{Re} \int_s^{T-\eps} \langle G(r)^*A(r)^* U(T-\eps,r)^*Q(T-\eps)U(T-\eps,s)x,u(r)\rangle_U\, dr \\
& & = 2 \textrm{Re} \langle Q(T-\eps)U(T-\eps,s)x,[L_su](T-\eps)\rangle_H\,.
\eeyy
In addition, 
\beyy 
\lefteqn{M_3 + N_3 + N_6 = -\int_s^{T-\eps} \|M(r)^{1/2}[L_su](r)\|_H^2\,dr}\\
& & - 2 \textrm{Re} \int_s^{T-\eps} \langle u(r),G(r)^*A(r)^*Q(r)[L_su](r)\rangle_U\,dr \\
& & + \int_s^{T-\eps} \|N(r)^{-1/2}G(r)^*A(r)^*Q(r)[L_su](r)\|_U^2\,dr \\
& & = -\int_s^{T-\eps} \langle [L^{*\eps}_s M(\cdot)L_su(\cdot)](r),u(r)\rangle_U\,dr \\
& & - 2\textrm{Re} \int_s^{T-\eps} \langle G(r)^*A(r)^*Q(r)[L_su](r),u(r)\rangle_U\,dr \\
& & + \int_s^{T-\eps} \langle [L^{*\eps}_s Q(\cdot)A(\cdot)G(\cdot)N(\cdot)^{-1}G(\cdot)^*A(\cdot)^*Q(\cdot)L_su](r),u(r)\rangle_U\,dr.
\eeyy
Putting together the first and third term, and using lemma \ref{calcoloLs}, we get 
\beyy 
\lefteqn{M_3 + N_3 +N_6 }\\
& & = - 2\textrm{Re} \int_s^{T-\eps} \langle G(r)^*A(r)^*Q(r)[L_su](r),u(r)\rangle_U\,dr  \\
& & -\int_s^{T-\eps} \langle L^{*\eps}_s [M(\cdot)-Q(\cdot)A(\cdot)G(\cdot)N(\cdot)^{-1}G(\cdot)^*A(\cdot)^*Q(\cdot)L_su](r),u(r)\rangle_U\,dr \\
& & = - 2\textrm{Re} \int_s^{T-\eps} \langle G(r)^*A(r)^*Q(r)[L_su](r),u(r)\rangle_U\,dr +\\
& & + 2 \textrm{Re} \int_s^{T-\eps} \Bigg\langle G(r)^*A(r)^*\Bigg[\int_r^{T-\eps} U(\sigma,r)^*\Big[M(\sigma)\\
& & \qquad -Q(\sigma)A(\sigma)G(\sigma)N(\sigma)^{-1}G(\sigma)^*A(\sigma)^*Q(\sigma)\Big] U(\sigma,r)\,d\sigma\Bigg] [L_su](r),u(r)\Bigg\rangle_U\,dr.
\eeyy
Recalling again \eqref{Riccintw}, we get
\beyy 
\lefteqn{M_3 + N_3 +N_6 }\\
& & \hspace{-3mm} = - 2\textrm{Re} \int_s^{T-\eps} \langle G(r)^*A(r)^*U(T-\eps,r)^*Q(T-\eps)U(T-\eps,r)[L_su](r),u(r)\rangle_U\,dr\\
& & \hspace{-3mm} = 2\textrm{Re} \int_s^{T-\eps}\hspace{-2mm} \int_s^r \langle G(r)^*A(r)^*U(T-\eps,r)^*Q(T-\eps) U(T-\eps,\sigma)A(\sigma)G(\sigma)u(\sigma)\,d\sigma, u(r)\rangle_U\,dr\\
& & \hspace{-3mm} = 2 \textrm{Re}\int_s^{T-\eps} \int_s^r \langle Q(T-\eps)U(T-\eps,\sigma)A(\sigma)G(\sigma)u(\sigma), U(T-\eps,r)A(r)G(r)u(r)\rangle_U\,d\sigma dr;
\eeyy
Since the last integrand is a symmetric function $b(\sigma,r)$, as in the proof of Lemma \ref{calcoloLs} we get
\beyy 
\lefteqn{M_3 + N_3 +N_6 =2 \textrm{Re}\int_s^{T-\eps} \int_s^r b(\sigma,r)\,d\sigma dr} \\
& & = \int_s^{T-\eps} \int_s^r b(\sigma,r)\,d\sigma dr + \int_s^{T-\eps} \int_s^r b(r,\sigma)\,d\sigma dr \\
& & = \int_s^{T-\eps} \int_s^r b(\sigma,r)\,d\sigma dr + \int_s^{T-\eps} \int_\sigma^{T-\eps} b(r,\sigma)\,drd\sigma =\int_s^{T-\eps} \int_s^{T-\eps}b(\sigma,r)\,d\sigma dr, 
\eeyy
so that finally
\beyy 
\lefteqn{M_3 + N_3 +N_6 }\\[2mm]
& & = \int_s^{T-\eps} \int_s^{T-\eps}\langle Q(T-\eps)U(T-\eps,\sigma)A(\sigma)G(\sigma)u(\sigma), U(T-\eps,r)A(r)G(r)u(r)\rangle_U\,d\sigma dr\\[2mm]
& & = \langle Q(T-\eps)[L_su](T-\eps),[L_su](T-\eps)\rangle_H\,.
\eeyy
Summing up, and recalling \eqref{state_Ls} and \eqref{Ls}, we have
\beyy
\lefteqn{M_1 + M_2 + M_3 + N_2 + N_3 + N_4 + N_5+N_6 }\\[2mm]
& & \hspace{-2mm} = -\langle Q(s)x,x\rangle_H + \langle Q(T-\eps)U(T-\eps,s)x,U(T-\eps,s)x\rangle_H \\[2mm]
& & \hspace{-2mm} \quad+ 2 \textrm{Re} \langle Q(T-\eps)U(T-\eps,s)x,[L_su](T-\eps)\rangle_H + \langle Q(T-\eps)[L_su](T-\eps),[L_su](T-\eps)\rangle_H\\[2mm]
& & \hspace{-2mm} = -\langle Q(s)x,x\rangle_H + \langle Q(T-\eps)y(T-\eps),y(T-\eps)\rangle_H\,,
\eeyy
which proves \eqref{idbase} and concludes the proof of Lemma \ref{idcontr}. \qed 
\section{Proof of Lemma \ref{Riccequiv}} \label{B}
In order to prove \eqref{Riccint1}, we start from the integral Riccati equation \eqref{Riccintw} and insert in place of $U(T-\eps,s)x$ and $U(r,s)x$ their expressions in terms of $\Phi(T-\eps,s)x$ and $\Phi(r,s)x$, given by \eqref{eqclosedloop}. All calculations are legitimate, since there are singularities only at $T$. Setting for simplicity
\beq\label{K} 
K(q,\sigma) = U(q,\sigma)A(\sigma)G(\sigma)N(\sigma)^{-1} G(\sigma)^*A(\sigma)^*Q(\sigma),
\eeq
we get
$$
\begin{array}{l}\langle Q(s)x,y\rangle_H=\langle Q(T-\eps)\Phi(T-\eps,s)x,U(T-\eps,s)y\rangle_H \\[2mm]
\displaystyle \quad + \left\langle Q(T-\eps)\int_s^{T-\eps} K(T-\eps,\sigma)\Phi(\sigma,s)x\,d\sigma,U(T-\eps,s)y\right\rangle_H \\[4mm]
\displaystyle \quad+\int_s^{T-\eps}\hspace{-1mm} \langle M(r) \Phi(r,s)x,U(r,s)y\rangle_H \,dr \hspace{-1mm}+\hspace{-1mm}\int_s^{T-\eps}\hspace{-1mm}\left\langle M(r) \hspace{-1mm}\int_s^r  \hspace{-1mm}K(r,\sigma)\Phi(\sigma,s)x\, d\sigma,U(r,s)y\right\rangle_H\, \hspace{-1mm}dr \\[4mm]
\displaystyle \quad-\int_s^{T-\eps}  \hspace{-1mm}\left\langle N(r)^{-1} G(r)^*A(r)^*Q(r)\Phi(r,s)x,G(r)^*A(r)^*Q(r) U(r,s)y\right\rangle_H\,dr \\[4mm]
\displaystyle \quad-\int_s^{T-\eps} \left\langle N(r)^{-1} G(r)^*A(r)^*Q(r) \hspace{-1mm}\int_s^r  \hspace{-1mm}K(r,\sigma)\Phi(\sigma,s)x\, d\sigma,G(r)^*A(r)^*Q(r) U(r,s)y\right\rangle_H  \hspace{-1mm}dr\\[4mm]
\displaystyle \ =\langle Q(T-\eps)\Phi(T-\eps,s)x,U(T-\eps,s)y\rangle_H +\int_s^{T-\eps} \langle M(r) \Phi(r,s)x,U(r,s)y\rangle_H \,dr +\\
\displaystyle \quad + I_1+I_2+I_3+I_4,
\end{array}$$
where $I_1,I_2,I_3,I_4$ are the second term and the last three ones of the right member of the first equality. We observe now that if we set
$$g(r):=A(r)G(r)N(r)^{-1} G(r)^*A(r)^*Q(r)\Phi(r,s)x,\quad h(r):=U(r,s)y,$$
then, using \eqref{K} and Fubini-Tonelli's Theorem, we can rewrite the sum of these terms as 
$$\begin{array}{l} 
\displaystyle I_1+I_2+I_3+I_4=\int_s^{T-\eps} \Bigg[ \langle Q(T-\eps)U(T-\eps,\sigma)g(\sigma),U(T-\eps,\sigma)h(\sigma)\rangle_H \\[5mm]
\displaystyle \qquad + \int_\sigma^{T-\eps}\langle M(r)U(r,\sigma)g(\sigma),U(r,\sigma)h(\sigma)\rangle_H\,dr - \langle Q(\sigma)g(\sigma),h(\sigma) \rangle_H \\[4mm]
\displaystyle \qquad - \int_\sigma^{T-\eps}\langle N(r)^{-1} G(r)^*A(r)^*Q(r)U(r,\sigma) g(\sigma),G(r)^*A(r)^*Q(r)U(r,\sigma)h(\sigma)\rangle_U \,dr\Bigg]d\sigma \hspace{-1mm}\\
\displaystyle \ = \int_s^{T-\eps} 0\,d\sigma = 0,
\end{array}$$
where the penultimate equality follows by \eqref{Riccintw}. This proves equation \eqref{Riccint1}.\\
We now prove \eqref{Riccint2}: in a quite similar way, we insert in \eqref{Riccint1}, in place of $U(T-\eps,s)x$ and $U(r,s)x$, their expressions in terms of $\Phi(T-\eps,s)x$ and $\Phi(r,s)x$, given by \eqref{eqclosedloop}: we have 
\beyy
\lefteqn{\langle Q(s)x,y\rangle_H=  \langle Q(T-\eps)\Phi(T-\eps,s)x,\Phi(T-\eps,s)y\rangle_H }\\[2mm]
& & \quad + \left\langle Q(T-\eps)\Phi(T-\eps,s)x,\int_s^{T-\eps} K(T-\eps,\sigma)\Phi(\sigma,s)y\,d\sigma\right\rangle_H \\
& &  \quad + \int_s^{T-\eps}\hspace{-1mm} \langle M(r)\Phi(r,s)x,\Phi(r,s)y\rangle_H\,dr+\hspace{-1mm} \int_s^{T-\eps} \hspace{-1mm}\left\langle M(r)\Phi(r,s)x,\hspace{-1mm}\int_s^rK(r,\sigma)\Phi(\sigma,s)y \hspace{-1mm}\right\rangle_H\,\hspace{-1mm}d\sigma \\
& &=\langle Q(T-\eps)\Phi(T-\eps,s)x,\Phi(T-\eps,s)y\rangle_H +  \int_s^{T-\eps} \langle M(r)\Phi(r,s)x,\Phi(r,s)y\rangle_H\,dr+ J_1+J_2,
\eeyy
where $J_1$ and $J_2$ are the second and fourth term of the right member of the first equality. Recalling \eqref{K}, we set
$$g(\sigma)=\Phi(\sigma,s)x, \qquad h(\sigma) = A(\sigma)G(\sigma)N(\sigma)^{-1}G(\sigma)^*A(\sigma)^*Q(\sigma)\Phi(\sigma,s)y,$$  
and using Fubini-Tonelli's Theorem we can write
\beyy
J_1+J_2  & = & \int_s^{T-\eps} \Bigg[ \langle Q(T-\eps)\Phi(T-\eps,\sigma) g(\sigma),U(T-\eps,\sigma)h(\sigma)\rangle_H \\
&  &\quad \quad \ + \int_\sigma^{T-\eps} \langle M(r)\Phi(r,\sigma) g(\sigma),U(r,\sigma) h(\sigma)\rangle_H\,dr\Bigg] d\sigma \\
& = & \int_s^{T-\eps} \langle Q(\sigma)g(\sigma),h(\sigma)\rangle_H\,d\sigma \\
& = & \int_s^{T-\eps} \langle N(\sigma)^{-1} G(\sigma)^*A(\sigma)^*Q(\sigma)\Phi(\sigma,s)x,  G(\sigma)^*A(\sigma)^*Q(\sigma)\Phi(\sigma,s)y\rangle_H\,d\sigma,
\eeyy
where in the penultimate equality we have used \eqref{Riccint1}. This proves equation \eqref{Riccint2} and concludes the proof. \qed

\end{document}